\documentclass[12pt,a4paper,british]{article}
\usepackage{amscd}
\usepackage{amsmath}
\usepackage{amsfonts}
\usepackage[pdftex]{graphicx}
\usepackage{latexsym}
\usepackage{amssymb}
\usepackage{cite}
\usepackage{theorem}
\usepackage{hyperref}
\usepackage{color}
\usepackage{cancel}
\usepackage{babel}
\usepackage[useregional]{datetime2}

\definecolor{airforceblue}{rgb}{0.36,0.54,0.66}
\definecolor{auburn}{rgb}{0.43,0.21,0.1}

\newcommand{\A}{\mathbb{A}}
\newcommand{\bA}{\boldsymbol{A}}
\newcommand{\bd}{\boldsymbol{d}}
\newcommand{\be}{\boldsymbol{e}}
\newcommand{\bg}{\boldsymbol{g}}
\newcommand{\bN}{\boldsymbol{N}}
\newcommand{\bq}{\boldsymbol{q}}
\newcommand{\bT}{\boldsymbol{T}}
\newcommand{\btau}{\boldsymbol{\tau}}
\newcommand{\bv}{\boldsymbol{v}}
\newcommand{\bV}{\boldsymbol{V}}
\newcommand{\bw}{\boldsymbol{w}}
\newcommand{\bX}{\boldsymbol{X}}
\newcommand{\bXi}{\boldsymbol{\Xi}}
\newcommand{\bza}{\boldsymbol{\alpha}}
\newcommand{\bzb}{\boldsymbol{\beta}}
\newcommand{\bzero}{\boldsymbol{0}}
\newcommand{\bzs}{\boldsymbol{\sigma}}
\newcommand{\C}{\mathbb{C}}
\newcommand{\CC}{\mathcal{C}}
\newcommand{\ClD}{\mathcal{D}}
\newcommand{\CF}{\mathcal{F}}
\newcommand{\CQ}{\mathcal{Q}}
\newcommand{\CR}{\mathcal{R}}
\newcommand{\CT}{\mathcal{T}}
\newcommand{\CW}{\mathcal{W}}
\newcommand{\dn}{{\text{\tiny dn}}}
\newcommand{\hf}{\frac{1}{2}}
\newcommand{\hff}{\tfrac{1}{2}}
\newcommand{\K}{\mathbb{K}}
\newcommand{\ozg}{\overline{\gamma}}
\newcommand{\R}{\mathbb{R}}
\newcommand{\tcb}{\textcolor{blue}}
\newcommand{\tk}{\tilde{k}}
\newcommand{\tS}{\tilde{S}}
\newcommand{\tw}{\tilde{w}}
\newcommand{\up}{{\text{\tiny up}}}
\newcommand{\za}{\alpha}
\newcommand{\zg}{\gamma}
\newcommand{\zk}{\kappa}
\newcommand{\zl}{\lambda}
\newcommand{\zL}{\Lambda}
\newcommand{\zo}{\omega}
\newcommand{\zO}{\Omega}
\newcommand{\zs}{\sigma}
\newcommand{\zS}{\Sigma}
\newcommand{\zz}{\zeta}
\newtheorem{theorem}{\bf Theorem}
\numberwithin{theorem}{section}

\newenvironment{proof}
{\begin{trivlist}\item[\hskip\labelsep\quad{\bf Proof.}
\hspace{0.5 em}]}{\hfill \rule{0.5em}{0.5em} \end{trivlist}}
\numberwithin{equation}{section}
\numberwithin{figure}{section}

\author{Cornelis van der Mee\tcb{\footnote{Dip. Matematica e Informatica,
Universit\`a di Ca\-glia\-ri, Via Ospedale 72, 09124 Ca\-glia\-ri, Italy. Email:
cornelis110553@gmail.com}}}
\title{Matrix Zakharov-Shabat Systems\\ with Zero Diagonal
Entry\tcb{\footnote{LaTeX compilation date and time: \tcb{\DTMnow}}}}

\begin{document}
\date{}
\maketitle

\begin{abstract}
In this article we develop the direct and inverse scattering theory of the
Ablowitz-Kaup-Newell-Segur (AKNS) system $\bv_x=(ik\zS+\CQ(x))\bv$, where $\zS$
is a diagonal $n\times n$ matrix with diagonal entries $1$ and $-1$ and a single
zero diagonal entry and $\CQ(x)$ is an $n\times n$ potential anticommuting with
$\zS$ with entries in $L^1(\R)$. We derive the time evolution of the scattering
data which, through the inverse scattering transform, lead to the solution of
the initial-value problem for a system of long-wave-short-wave equations.
\end{abstract}



\section{Introduction}\label{sec:1}

The direct and inverse scattering theory of the AKNS system
\begin{equation}\label{1.1}
\bv_x=(ik\zS+\CQ)\bv,
\end{equation}
where $\zS$ is a diagonal matrix of order $n$ with nonzero real diagonal
entries, $\CQ(x)$ is an $n\times n$ matrix potential anticommuting with $\zS$
with its entries belonging to $L^1(\R)$ and $x\in\R$ is position, has been
studied at great length \cite{ZS,AKNS,AC,AS,CD,FT,APT}. Its direct scattering
theory relies primarily on the linear relation between $n$ linearly independent
Jost solutions continuous in $k\in\C^+\cup\R$ and analytic in $k\in\C^+$ and $n$
linearly independent Jost solutions continuous in $k\in\C^-\cup\R$ and analytic
in $k\in\C^-$, where $\zS$ is a real diagonal matrix of order $n$. The
scattering data resulting from the direct scattering theory consist of two sets
of reflection coefficients and as many norming constants as there are discrete
eigenvalues in $\C^+$ or in $\C^-$. Its inverse scattering theory consists of
converting the scattering data in a Marchenko integral equation whose solution
leads to the potential $\CQ(x)$. An alternative way of implementing inverse
scattering consists of converting the scattering data in a Riemann-Hilbert
problem whose solution leads to the potential $\CQ(x)$. The direct and inverse
scattering theory of the AKNS system \eqref{1.1} with nonsingular real diagonal
matrix $\zS$ has been worked out in various textbooks \cite{AC,CD,FT,APT}.
A general theory of developing the direct and inverse scattering for arbitrary
nonsingular complex diagonal matrix $\zS$ can be found in \cite{BC1,BC2,BDT}.

 Few publications are devoted to AKNS systems \eqref{1.1}, where $\zS$ is a real
diagonal matrix with at least one of its diagonal entries vanishing. Among the
very few papers devoted to the subject matter is that by Newell \cite{Nw}, where
$\zS=\text{diag}(1,0,-1)$. In this article the author defines two Jost solutions
analytic in $k\in\C^+$ and two Jost solutions analytic in $k\in\C^-$.
Unfortunately, one would expect three Jost solutions analytic in $k\in\C^+$ and
three Jost solutions analytic in $k\in\C^-$, thus rendering incomplete the
direct scattering theory expounded in \cite{Nw}. In a recent paper \cite{DM24}
we have solved this problem for $n=3$ and $\zS=\text{diag}(1,0,-1)$. In the
present article we shall extend this result to arbitrary $n\ge3$ and
$\zS=\text{diag}(I_{m_+},0,-I_{m_-})$ with $\min(m_+,m_-)\ge1$.

In this article we develop a direct and inverse scattering theory of the AKNS
system \eqref{1.1}, where
$$\zS=\text{diag}(1,\ldots,1,0,-1,\ldots,-1)$$
with $m_+\ge1$ diagonal entries $+1$, one diagonal entry zero, $m_-\ge1$
diagonal entries $-1$, and $m_++m_-=n-1$. To define the $n$-th linearly
independent Jost solution analytic in either $k\in\C^+$ or in $k\in\C^-$,
we take the wedge products of the complex conjugates of the $n-1$ dual Jost
solutions satisfying $\breve{\bv}_x=-\breve{\bv}(ik\zS+\CQ)$, where we adopt
an idea applied before to the Manakov system with $\zS=\text{diag}(1,-1,-1)$
with nonvanishing boundary conditions \cite{PAB}. Similar constructions of
an $n$-th Jost-like solution have been given before by Deift et al. \cite{DTT},
Kaup \cite{Kp2}, and McKean \cite{McK}. Once we have defined $n$ linearly
independent Jost solutions analytic in $k\in\C^+$ written as the columns of
an $n\times n$ matrix and $n$ linearly independent Jost solutions analytic in
$k\in\C^-$ written as the columns of another $n\times n$ matrix, we write the
first $n\times n$ matrix as the other multiplied from the right by
an $n\times n$ matrix depending only on $k\in\R$, where the diagonal entries are
called the transmission coefficients and, apart from minus signs, the
off-diagonal entries are called reflection coefficients. We thus find the
reflection coefficients and, corresponding to each pole of the transmission
coefficients in $\C^+$ or in $\C^-$, the so-called norming constants. The
inverse scattering problem can then be solved alternatively by either the
Marchenko method or the Riemann-Hilbert method. The direct and inverse
scattering theory of the AKNS system \eqref{1.1}, where $\zS$ contains at least
two zero diagonal entries and the other diagonal entries are $\pm1$, remains
open.

The study of one-dimensional wave propagation poses several challenging
mathematical problems \cite{Wh,Ln}. The most elementary mathematical model for
studying the interaction between long waves and short waves is the AKNS system
studied by Newell \cite{Nw}. Similar integrable equations were formulated by
Yan-Chow Ma \cite{Ma}, Yajima and Oikawa \cite{YO}, Ruomeng Li and Xianguo Geng
\cite {Geng3}, and Wang et al. \cite{WCGL}. More recently, Caso-Huerta et al.
\cite{YON,YON2} have studied more general long-wave-short-wave interaction
models. These so-called YON equations can be derived from the zero curvature
condition
\begin{equation}\label{1.2}
\bX_t-\bT_x+\bX\bT-\bT\bX=0_{3\times3},
\end{equation}
where the Lax pair $(\bX,\bT)$ is given by \cite[Eqs.~(2.2), (3.1), and
(3.2)]{YON} as well as by \cite[Eqs.~(7)]{YON2}. The nonlinear wave equation
systems studied in \cite{YON,YON2} have been generalized by Wright
\cite[Eqs.~(52)-(53)]{Wr}.

In the present $n\times n$ situation the Lax pair $(\bX,\bT)$, where
$\bX=ik\zS+\CQ$ with potential $\CQ=\left[\begin{smallmatrix}\bzero&S&iL\\
T^\dagger&0&S^\dagger\\iK&T&\bzero\end{smallmatrix}\right]$ and
$\bT=(ik)^2A+ikB+C$ with
\begin{align*}
A&=\frac{i}{n}\text{diag}(I_{m_+},1-n,I_{m_-}),\\
B&=\begin{bmatrix}\bzero&iS&\bzero\\iT^\dagger&0&-iS^\dagger\\ \bzero&-iT
&\bzero\end{bmatrix},\quad C=\begin{bmatrix}-iST^\dagger&iS_x-LT&iSS^\dagger\\
-T_x^\dagger-S^\dagger K&i[T^\dagger S+S^\dagger T]&-iS_x^\dagger-T^\dagger L\\
iTT^\dagger&iT_x-KS&-iTS^\dagger\end{bmatrix},
\end{align*}
leads to the six coupled long-wave-short-wave equations
\begin{subequations}\label{1.3}
\begin{align}
iL_t-2i(SS^\dagger)_x-2ST^\dagger L+2LTS^\dagger&=0_{m_+\times m_-},
\label{1.3a}\\
iK_t-2i(TT^\dagger)_x-2TS^\dagger K+2KST^\dagger&=0_{m_-\times m_+},
\label{1.3b}\\
T_t-iT_{xx}+K_xS-iKLT+2iTS^\dagger T&=0_{m_-\times1},\label{1.3c}\\
S_t-iS_{xx}+L_xT-iLKS+2iST^\dagger S&=0_{m_+\times1},\label{1.3d}\\
T_t^\dagger+iT_{xx}^\dagger+S^\dagger K_x+iT^\dagger LK-2iT^\dagger ST^\dagger
&=0_{1\times m_+},\label{1.3e}\\
S_t^\dagger+iS_{xx}^\dagger+T^\dagger L_x+iS^\dagger KL-2iS^\dagger TS^\dagger
&=0_{1\times m_-},\label{1.3f}
\end{align}
\end{subequations}
where $L$ and $K$ are real symmetric long-wave square matrix functions and $T$
and $S$ are complex short-wave column vector functions. Note that \eqref{1.3e}
and \eqref{1.3f} are the complex conjugate transposes of \eqref{1.3c} and
\eqref{1.3d}. In \cite{YON,YON2} the special case where $n=3$, $T=\za S$, and
$K=\za^2L$ for some $\za>0$, is treated.

Let us discuss the contents of the various sections. In Sections
\ref{sec:2}--\ref{sec:4} we discuss the scattering solutions of the AKNS system
$\bv_x=(ik\zS+\CQ)\bv$, those of the dual AKNS system
$\breve{\bv}_x=-\breve{\bv}(ik\zS+\CQ)$, and the reflection and transmission
coefficients. Sections \ref{sec:5} and \ref{sec:6} are devoted to the Marchenko
integral equations and the time evolution of the scattering data. We add
Appendix~\ref{sec:A} on linear equations for wedge products. We also include
Appendix~\ref{sec:B} from \cite{DM24} on Wiener algebra properties of scattering
solutions and scattering coefficients. Appendix~C from \cite{DM24} on the
essential spectrum of the AKNS system can be copied almost verbatim from
\cite{DM24} to the present paper and will not included here.

\section{Jost solutions}\label{sec:2}

Let us define the Jost matrices $\Psi(x,k)$ and $\Phi(x,k)$ as those $n\times n$
matrix solutions of the AKNS system \eqref{1.1} satisfying the asymptotic
conditions
\begin{subequations}\label{2.1}
\begin{alignat}{3}
\Psi(x,k)&=e^{ikx\zS}[I_n+o(1)],&\qquad&x\to+\infty,\label{2.1a}\\
\Phi(x,k)&=e^{ikx\zS}[I_n+o(1)],&\qquad&x\to-\infty,\label{2.1b}
\end{alignat}
\end{subequations}
where $\det\Psi(x,k)=\det\Phi(x,k)=e^{ikx(m_+-m_-)}$ for $(x,k)\in\R^2$. Using
that the entries of the off-diagonal matrix $\CQ(x)$ belong to $L^1(\R)$, we
easily arrive at the Volterra integral equations
\begin{subequations}\label{2.2}
\begin{align}
\Psi(x,k)&=e^{ikx\zS}-\int_x^\infty dy\,e^{-ik(y-x)\zS}\CQ(y)\Psi(y,k),
\label{2.2a}\\
\Phi(x,k)&=e^{ikx\zS}+\int_{-\infty}^x dy\,e^{ik(x-y)\zS}\CQ(y)\Phi(y,k).
\label{2.2b}
\end{align}
\end{subequations}
In terms of the Faddeev functions $M(x,k)=\Psi(x,k)e^{-ikx\zS}$ and
$N(x,k)=\Phi(x,k)e^{-ikx\zS}$, we obtain the Volterra integral equations
\begin{subequations}\label{2.3}
\begin{align}
M(x,k)&=I_n-\int_x^\infty dy\,e^{-ik(y-x)\zS}\CQ(y)M(y,k)e^{ik(y-x)\zS},
\label{2.3a}\\
N(x,k)&=I_n+\int_{-\infty}^x dy\,e^{ik(x-y)\zS}\CQ(y)N(y,k)e^{-ik(x-y)\zS}.
\label{2.3b}
\end{align}
\end{subequations}

Letting $\bzs$ stand for a hermitian and unitary $n\times n$ matrix, we call the
potential $\bzs$-{\it focusing} if
\begin{equation}\label{2.4}
\CQ(x)^\dagger=-\bzs\CQ(x)\bzs,
\end{equation}
where the dagger stands for the matrix conjugate transpose. If $\bzs=I_n$, the
potential is called {\it focusing}. It is easily verified that for
$\bzs$-focusing potentials and $(x,k)\in\R^2$
\begin{subequations}\label{2.5}
\begin{alignat}{3}
\Psi(x,k)^\dagger&=\bzs\Psi(x,k)^{-1}\bzs,&\quad
\Phi(x,k)^\dagger&=\bzs\Phi(x,k)^{-1}\bzs,\label{2.5a}\\
M(x,k)^\dagger&=\bzs M(x,k)^{-1}\bzs,&\quad
N(x,k)^\dagger&=\bzs N(x,k)^{-1}\bzs,\label{2.5b}
\end{alignat}
\end{subequations}
whenever $\bzs$ and $\zS$ commute. On the contrary,
\begin{subequations}\label{2.6}
\begin{alignat}{3}
\Psi(x,k)^\dagger&=\bzs\Psi(x,-k)^{-1}\bzs,&\quad
\Phi(x,k)^\dagger&=\bzs\Phi(x,-k)^{-1}\bzs,\label{2.6a}\\
M(x,k)^\dagger&=\bzs M(x,-k)^{-1}\bzs,&\quad
N(x,k)^\dagger&=\bzs N(x,-k)^{-1}\bzs,\label{2.6b}
\end{alignat}
\end{subequations}
whenever $\bzs$ and $\zS$ anticommute. We observe that the $n\times n$ matrices
$\bzs$ having diagonal blocks of the respective orders $m_+$, $1$, and $m_-$,
have the form
\begin{equation}\label{2.7}
\bzs=\underbrace{\begin{bmatrix}\bzs_{++}&\bzero&\bzero\\ \bzero&\zs_{00}
&\bzero\\\bzero&\bzero&\bzs_{--}\end{bmatrix}}_{\text{if}\ \bzs\zS=\zS\bzs},
\qquad\bzs=\underbrace{\begin{bmatrix}\bzero&\bzero&\bzs_{+-}\\
\bzero&\zs_{00}&\bzero\\ \bzs_{-+}&\bzero&\bzero\end{bmatrix}}_{\text{if}\
\bzs\zS=-\zS\bzs\ \text{and}\ m_+=m_-}.
\end{equation}
Moreover, if $\bzs$ is hermitian and unitary and commutes with $\zS$, then
$\bzs_{++}$ and $\bzs_{--}$ are hermitian and unitary and $\zs_{00}=\pm1$. On
the other hand, if $\bzs$ is hermitian and unitary and anticommutes with $\zS$,
then $\bzs_{+-}$ and $\bzs_{-+}$ are unitary and each other's adjoints (and
hence $m_+=m_-=\tfrac{1}{2}(n-1)$) and $\zs_{00}=\pm1$.

Define the following submatrices of the $n\times n$ identity matrix $I_n$:
$$\be_+=\begin{bmatrix}I_{m_+}\\0_{(1+m_-)\times m_+}\end{bmatrix},\quad
\be_0=\begin{bmatrix}0_{m_+}\\1\\0_{m_-}\end{bmatrix},\quad
\be_-=\begin{bmatrix}0_{(m_++1)\times m_-}\\I_{m_-}\end{bmatrix}.$$

\begin{theorem}\label{th:2.1}
There exist $n\times n$ matrices $K(x,y)$ and $J(x,y)$ such that the triangular
representations
\begin{subequations}\label{2.8}
\begin{align}
M(x,k)\be_+&=\be_++\int_x^\infty dy\,e^{ik(y-x)}K(x,y)\be_+,\label{2.8a}\\
M(x,k)\be_-&=\be_-+\int_x^\infty dy\,e^{-ik(y-x)}K(x,y)\be_-,\label{2.8b}\\
N(x,k)\be_+&=\be_++\int_{-\infty}^x dy\,e^{-ik(x-y)}J(x,y)\be_+,\label{2.8c}\\
N(x,k)\be_-&=\be_-+\int_{-\infty}^x dy\,e^{ik(x-y)}J(x,y)\be_-,\label{2.8d}
\end{align}
\end{subequations}
are true, where the integrals
\begin{subequations}\label{2.9}
\begin{align}
\int_x^\infty dy\,\|K(x,y)\be_+\|+\int_{-\infty}^x dy\,\|J(x,y)\be_+\|
\label{2.9a}\\
\int_x^\infty dy\,\|K(x,y)\be_-\|+\int_{-\infty}^x dy\,\|J(x,y)\be_-\|,
\label{2.9b}
\end{align}
\end{subequations}
converge uniformly in $x\in\R$.
\end{theorem}

\begin{proof}
Define the $n\times n$ diagonal matrices
$$\bN_+=2I_{m_+}\oplus 1\oplus 0_{m_-},\qquad
\bN_-=0_{m_+}\oplus 1\oplus2I_{m_-}.$$
Postmultiplying \eqref{2.3} by $\be_+$ or $\be_-$, we obtain
\begin{alignat*}{3}
M(x,k)\be_+&=\be_+-\int_x^\infty dy\,e^{ik(y-x)\bN_-}\CQ(y)M(y,k)\be_+,\\
M(x,k)\be_-&=\be_--\int_x^\infty dy\,e^{-ik(y-x)\bN_+}\CQ(y)M(y,k)\be_-,\\
N(x,k)\be_+&=\be_++\int_{-\infty}^x dy\,e^{-ik(x-y)\bN_-}\CQ(y)N(y,k)\be_+,\\
N(x,k)\be_-&=\be_-+\int_{-\infty}^x dy\,e^{ik(x-y)\bN_+}\CQ(y)N(y,k)\be_-.
\end{alignat*}
Substituting \eqref{2.8} in \eqref{2.3} and stripping off the Fourier transform
we get
\begin{align*}
\be_+^TK(x,y)\be_+&=-\int_x^\infty dz\,\be_+^T\CQ(z)K(z,z+y-x)\be_+,\\
\be_0^TK(x,y)\be_+&=-\be_0^T\CQ(y)\be_+-\int_x^y dz\,\be_0^T\CQ(z)K(z,y)\be_+,\\
\be_-^TK(x,y)\be_+&=-\hff\be_-^T\CQ(\hff[x+y])\be_+
-\int_x^{\hff[x+y]}dz\,\be_-^T\CQ(z)K(z,x+y-z)\be_+,\\
\be_+^TK(x,y)\be_-&=-\hff\be_+^T\CQ(\hff[x+y])\be_-
-\int_x^{\hff[x+y]}dz\,\be_+^T\CQ(z)K(z,x+y-z)\be_-,\\
\be_0^TK(x,y)\be_-&=-\be_0^T\CQ(y)\be_--\int_x^y dz\,\be_0^T\CQ(z)K(z,y)\be_-,\\
\be_-^TK(x,y)\be_-&=-\int_x^\infty dz\,\be_-^T\CQ(z)K(z,z+y-x)\be_-,\\
\be_+^TJ(x,y)\be_+&=\int_{-\infty}^x dz\,\be_+^T\CQ(z)J(z,z+y-x)\be_+,\\
\be_0^TJ(x,y)\be_+&=\be_0^T\CQ(y)\be_++\int_y^x dz\,\be_0^T\CQ(z)J(z,y)\be_+,
\end{align*}
\begin{align*}
\be_-^TJ(x,y)\be_+&=\hff\be_-^T\CQ(\hff[x+y])\be_+
+\int_{\hff[x+y]}^xdz\,\be_-^T\CQ(z)J(z,x+y-z)\be_+,\\
\be_+^TJ(x,y)\be_-&=\hff\be_+^T\CQ(\hff[x+y])\be_-
+\int_{\hff[x+y]}^xdz\,\be_+^T\CQ(z)J(z,x+y-z)\be_-,\\
\be_0^TJ(x,y)\be_-&=\be_0^T\CQ(y)\be_-+\int_y^x dz\,\be_0^T\CQ(z)J(z,y)\be_-,\\
\be_-^TJ(x,y)\be_-&=\int_{-\infty}^x dz\,\be_-^T\CQ(z)J(z,z+y-x)\be_-.
\end{align*}
Since each entry of $\CQ(x)$ belongs to $L^1(\R)$, we obtain \eqref{2.9} with
the help of Gronwall's inequality.
\end{proof}

 For $x=y$ we obtain the following expressions yielding the potentials:
\begin{subequations}\label{2.10}
\begin{align}
\be_+^TK(x,x)\be_+&=-\int_x^\infty dz\,\be_+^T\CQ(z)K(z,z)\be_+,\label{2.10a}\\
\be_0^TK(x,x)\be_+&=-\be_0^T\CQ(x)\be_+=-\bq_{0+}(x),\label{2.10b}\\
\be_-^TK(x,x)\be_+&=-\hf\be_-^T\CQ(x)\be_+=-\hf\bq_{-+}(x),\label{2.10c}\\
\be_+^TK(x,x)\be_-&=-\hf\be_+^T\CQ(x)\be_-=-\hf\bq_{+-}(x),\label{2.10d}\\
\be_0^TK(x,x)\be_-&=-\be_0^T\CQ(x)\be_-=-\bq_{0-}(x),\label{2.10e}\\
\be_-^TK(x,x)\be_-&=-\int_x^\infty dz\,\be_-^T\CQ(z)K(z,z)\be_-,\label{2.10f}\\
\be_+^TJ(x,x)\be_+&=\int_{-\infty}^x dz\,\be_+^T\CQ(z)K(z,z)\be_+,
\label{2.10g}\\
\be_0^TJ(x,x)\be_+&=\be_0^T\CQ(z)\be_+=\bq_{0+}(x),\label{2.10h}\\
\be_-^TJ(x,x)\be_+&=\hf\be_-^T\CQ(x)\be_+=\hf\bq_{-+}(x),\label{2.10i}\\
\be_+^TJ(x,x)\be_-&=\hf\be_+^T\CQ(x)\be_-=\hf\bq_{+-}(x),\label{2.10j}\\
\be_0^TJ(x,x)\be_-&=\be_0^T\CQ(z)\be_-=\bq_{0-}(x),\label{2.10k}\\
\be_-^TJ(x,x)\be_-&=\int_{-\infty}^x dz\,\be_-^T\CQ(z)K(z,z)\be_-,\label{2.10l}
\end{align}
\end{subequations}
where
$$\CQ(x)=\begin{bmatrix}0_{m_+\times m_+}&\bq_{+0}(x)&\bq_{+-}(x)\\
\bq_{0+}(x)&0&\bq_{0-}(x)\\ \bq_{-+}(x)&\bq_{-0}(x)
&0_{m_-\times m_-}\end{bmatrix}.$$
These expressions allow us to compute $\CQ(x)$  from either $K(x,x)$ or $J(x,x)$
with the exception of the blocks $\bq_{+0}(x)$ and $\bq_{-0}(x)$. Substituting
them into the integrals on the right-hand side we get
\begin{subequations}\label{2.11}
\begin{align}
\be_+^TK(x,x)\be_+&=\int_x^\infty dz\left(\bq_{+0}(z)\bq_{0+}(z)
+\hf\bq_{+-}(z)\bq_{-+}(z)\right),\label{2.11a}\\
\be_-^TK(x,x)\be_-&=\int_x^\infty dz\left(\hf\bq_{-+}(z)\bq_{+-}(z)
+\bq_{-0}(z)\bq_{0-}(z)\right),\label{2.11b}\\
\be_+^TJ(x,x)\be_+&=\int_{-\infty}^x dz\left(\bq_{+0}(z)\bq_{0+}(z)
+\hf\bq_{+-}(z)\bq_{-+}(z)\right),\label{2.11c}\\
\be_-^TJ(x,x)\be_-&=\int_{-\infty}^x dz\left(\hf\bq_{-+}(z)\bq_{+-}(z)
+\bq_{-0}(z)\bq_{0-}(z)\right).\label{2.11d}
\end{align}
\end{subequations}
The last four expressions allow us to compute the remaining blocks $\bq_{+0}(x)$
and $\bq_{-0}(x)$ of $\CQ(x)$ from either $K(x,x)$ or $J(x,x)$, provided
$\bq_{0+}(x)$ and $\bq_{0-}(x)$ are nonsingular matrices. The remaining blocks
can also be computed from either $K(x,x)$ or $J(x,x)$ if the potential is
$\bzs$-focusing.

\section{Dual scattering solutions}\label{sec:3}

In the preceding section we have defined $2(n-1)$ out of the $2n$ Jost
solutions, derived their Volterra integral equations, and proved their
triangular representations. To define the remaining two Jost solutions, we first
need to introduce the so-called dual Jost solutions.

\subsection{Dual Jost solutions}\label{sec:3.1}

The dual Jost solutions $\breve{\Psi}(x,k)$ and $\breve{\Phi}(x,k)$ are defined
as the unique solutions of the Volterra integral equations
\begin{subequations}\label{3.1}
\begin{align}
\breve{\Psi}(x,k)&=e^{-ikx\zS}+\int_x^\infty dy\,\breve{\Psi}(y,k)\CQ(y)
e^{ik(y-x)\zS},\label{3.1a}\\
\breve{\Phi}(x,k)&=e^{-ikx\zS}-\int_{-\infty}^x dy\,\breve{\Phi}(y,k)\CQ(y)
e^{-ik(x-y)\zS},\label{3.1b}
\end{align}
\end{subequations}
where $(x,k)\in\R^2$. Then $\breve{\Psi}(x,k)$ and $\breve{\Phi}(x,k)$ satisfy
the first order system
\begin{equation}\label{3.2}
\breve{\bv}_x=-\breve{\bv}(ik\zS+\CQ)
\end{equation}
with asymptotic conditions
\begin{subequations}\label{3.3}
\begin{alignat}{3}
\breve{\Psi}(x,k)&=[I_n+o(1)]e^{-ikx\zS},&\qquad&x\to+\infty,\label{3.3a}\\
\breve{\Phi}(x,k)&=[I_n+o(1)]e^{-ikx\zS},&\qquad&x\to-\infty.\label{3.3b}
\end{alignat}
\end{subequations}
It is easily verified that
\begin{equation}\label{3.4}
\breve{\Psi}(x,k)=\Psi(x,k)^{-1},\qquad\breve{\Phi}(x,k)=\Phi(x,k)^{-1},
\end{equation}
where $(x,k)\in\R^2$. We now introduce the dual Faddeev functions
$$\breve{M}(x,k)=e^{ikx\zS}\breve{\Psi}(x,k),\qquad
\breve{N}(x,k)=e^{ikx\zS}\breve{\Phi}(x,k).$$
Then $\breve{M}(x,k)=M(x,k)^{-1}$ and $\breve{N}(x,k)=N(x,k)^{-1}$ satisfy the
Volterra integral equations
\begin{subequations}\label{3.5}
\begin{align}
\breve{M}(x,k)&=I_n+\int_x^\infty dy\,e^{-ik(y-x)\zS}\breve{M}(y,k)\CQ(y)
e^{ik(y-x)\zS},\label{3.5a}\\
\breve{N}(x,k)&=I_n-\int_{-\infty}^x dy\,e^{ik(x-y)\zS}\breve{N}(y,k)\CQ(y)
e^{-ik(x-y)\zS}.\label{3.5b}
\end{align}
\end{subequations}
Moreover, we have the triangular representations
\begin{subequations}\label{3.6}
\begin{align}
\be_+^T\breve{M}(x,k)&=\be_+^T+\int_x^\infty dy\,e^{-ik(y-x)}
\be_+^T\breve{K}(y,x),\label{3.6a}\\
\be_-^T\breve{M}(x,k)&=\be_-^T+\int_x^\infty dy\,e^{ik(y-x)}
\be_-^T\breve{K}(y,x),\label{3.6b}\\
\be_+^T\breve{N}(x,k)&=\be_+^T+\int_{-\infty}^x dy\,e^{ik(x-y)}
\be_+^T\breve{J}(y,x),\label{3.6c}\\
\be_-^T\breve{N}(x,k)&=\be_-^T+\int_{-\infty}^x dy\,e^{-ik(x-y)}
\be_-^T\breve{J}(y,x),\label{3.6d}
\end{align}
\end{subequations}
where
\begin{align*}
\int_x^\infty dy\,\|\be_+^T\breve{K}(y,x)\|
&+\int_{-\infty}^x dy\,\|\be_+^T\breve{J}(y,x)\|<+\infty,\\
\int_x^\infty dy\,\|\be_-^T\breve{K}(y,x)\|
&+\int_{-\infty}^x dy\,\|\be_-^T\breve{J}(y,x)\|<+\infty.
\end{align*}
The proof of these triangular representations is based on Volterra integral
equations for $\be_\pm^T\breve{K}(y,x)$ and $\be_\pm^T\breve{J}(x,y)$ which are
very similar to those for $K(x,y)\be_\pm$ and $J(x,y)\be_\pm$.

Let us consider the adjoint symmetries of the first and third columns of the
 Faddeev functions. If the potential is $\bzs$-focusing and $\bzs$ commutes with
$\zS$ and hence satisfies $\bzs=\bzs_{++}\oplus\zs_{00}\oplus\bzs_{--}$, we get
\begin{subequations}\label{3.7}
\begin{align}
[M(x,k^*)\be_+]^\dagger&=\bzs_{++}\be_+^T\breve{M}(x,k)\bzs,\label{3.7a}\\
[N(x,k^*)\be_+]^\dagger&=\bzs_{++}\be_+^T\breve{N}(x,k)\bzs,\label{3.7b}\\
[M(x,k^*)\be_-]^\dagger&=\bzs_{--}\be_-^T\breve{M}(x,k)\bzs,\label{3.7c}\\
[N(x,k^*)\be_-]^\dagger&=\bzs_{--}\be_-^T\breve{N}(x,k)\bzs.\label{3.7d}
\end{align}
\end{subequations}
Using \eqref{2.8} and \eqref{3.6} we obtain from \eqref{3.7}
\begin{subequations}\label{3.8}
\begin{alignat}{3}
[K(x,y)\be_+]^\dagger&=\bzs_{++}\be_+^T\breve{K}(y,x)\bzs,&\quad
[J(x,y)\be_+]^\dagger&=\bzs_{++}\be_+^T\breve{J}(y,x)\bzs,\label{3.8a}\\
[K(x,y)\be_-]^\dagger&=\bzs_{--}\be_-^T\breve{K}(y,x)\bzs,&\quad
[J(x,y)\be_-]^\dagger&=\bzs_{--}\be_-^T\breve{J}(y,x)\bzs.\label{3.8b}
\end{alignat}
\end{subequations}
On the other hand, if the potential is $\bzs$-focusing and $\bzs$ anticommutes
with $\zS$ and hence has the form \eqref{2.7} for $\bzs\zS=-\zS\bzs$, we get
\begin{subequations}\label{3.9}
\begin{align}
[M(x,k^*)\be_+]^\dagger&=\bzs_{+-}\be_-^T\breve{M}(x,-k)\bzs,\label{3.9a}\\
[N(x,k^*)\be_+]^\dagger&=\bzs_{+-}\be_-^T\breve{N}(x,-k)\bzs,\label{3.9b}\\
[M(x,k^*)\be_-]^\dagger&=\bzs_{-+}\be_+^T\breve{M}(x,-k)\bzs,\label{3.9c}\\
[N(x,k^*)\be_-]^\dagger&=\bzs_{-+}\be_+^T\breve{N}(x,-k)\bzs.\label{3.9d}
\end{align}
\end{subequations}
Using \eqref{2.8} and \eqref{3.6} we obtain from \eqref{3.9}
\begin{subequations}\label{3.10}
\begin{alignat}{3}
[K(x,y)\be_+]^\dagger&=\bzs_{+-}\be_-^T\breve{K}(y,x)\bzs,\quad
[J(x,y)\be_+]^\dagger&=\bzs_{+-}\be_-^T\breve{J}(y,x)\bzs,\label{3.10a}\\
[K(x,y)\be_-]^\dagger&=\bzs_{-+}\be_+^T\breve{K}(y,x)\bzs,\quad
[J(x,y)\be_-]^\dagger&=\bzs_{-+}\be_+^T\breve{J}(y,x)\bzs.\label{3.10b}
\end{alignat}
\end{subequations}

\subsection{The missing two Jost solutions}\label{sec:3.2}

The $n-1$ columns of the Faddeev matrices $M(x,k)\be_+$ and $N(x,k)\be_-$ yield
$n-1$ linearly independent solutions of the differential system
\begin{equation}\label{3.11}
W_x=(ik\zS+\CQ)W-ikW\zS
\end{equation}
that are continuous in $k\in\C^+\cup\R$, are analytic in $k\in\C^+$, and tend to
one of the columns of $I_n$ (with the exception of the $(m_++1)$-st column) as
$k\to\infty$ from within $\C^+\cup\R$. Analogously, the $n-1$ columns of the
 Faddeev matrices $N(x,k)\be_+$ and $M(x,k)\be_-$ yield $n-1$ linearly
independent solutions of the differential system \eqref{3.11} that are
continuous in $k\in\C^-\cup\R$, are analytic in $k\in\C^-$, and tend to one of
the columns of $I_n$ (with the exception of the $(m_++1)$-st column) as
$k\to\infty$ from within $\C^-\cup\R$. Our objective is to construct the missing
two Faddeev functions with the correct analyticity properties.

Let $\breve{\Psi}(x,k)$ and $\breve{\Phi}(x,k)$ denote the dual Jost solutions.
Then their rows $\be_j^T\breve{\Psi}(x,k)$ and $\be_j^T\breve{\Phi}(x,k)$
satisfy \eqref{3.2}. Thus their transposes $\breve{\Psi}(x,k)^T\be_j$ and
$\breve{\Phi}(x,k)^T\be_j$ satisfy the differential system
$$\breve{\bv}^T_x=-(ik\zS+\CQ^T)\breve{\bv}^T.$$
Hence for each constant $\za\in\R$ the function
$\breve{\bw}(x)=e^{\za x}\breve{\bv}(x)^T$ satisfies the differential system
$$\breve{\bw}_x=(\za I_n-ik\zS-\CQ^T)\breve{\bw}=B\breve{\bw},$$
where $B=\za I_n-ik\zS-\CQ^T$. To get $\breve{\bw}$ to satisy the AKNS system
\eqref{1.1} we need to choose $\za$ in such a way that the identity 
$$(\text{Tr}\,B)I_n-B^T=ik\zS+\CQ$$
is satisfied. To do so we choose $\za=\text{Tr}\,B=\tfrac{ik(m_+-m_-)}{n-1}$.
Thus if $\breve{\bv}^1,\ldots,\breve{\bv}^{n-1}$ are $n-1$ linearly independent
row vector solutions of the AKNS system $\breve{\bv}_x=-\breve{\bv}(ik\zS+\CQ)$
and $\za=\tfrac{ik(m_+-m_-)}{n-1}$, then the wedge product
$$(e^{\za x}{\breve{\bv}}^1)^T\wedge\ldots\wedge
(e^{\za x}{\breve{\bv}}^{n-1})^T,$$
is a vector solution of the AKNS system \eqref{1.1}, provided we write this
wedge product with respect to the basis $\{\bd_1,\ldots,\bd_n\}$ of
$\zL^{n-1}(\R^n)$ [See Appendix~\ref{sec:A}]. Hence, with respect to the basis
$\{\bd_1,\ldots,\bd_n\}$ of $\zL^{n-1}(\R^n)$, the vector functions
\begin{subequations}\label{3.12}
\begin{align}
\zO_+(x,k)&\simeq(-1)^{m_+m_-}\left(e^{\za x}\breve{\Phi}(x,k)^T\be_1\wedge
\ldots\wedge e^{\za x}\breve{\Phi}(x,k)^T\be_{m_+}\nonumber\right.\\
&\left.\wedge e^{\za x}\breve{\Psi}(x,k)^T\be_{n+1-m_-}\wedge
\ldots\wedge e^{\za x}\breve{\Psi}(x,k)^T\be_n\right),\label{3.12a}\\
\zO_-(x,k)&\simeq(-1)^{m_+m_-}\left(e^{\za x}\breve{\Psi}(x,k)^T\be_1\wedge
\ldots\wedge e^{\za x}\breve{\Psi}(x,k)^T\be_{m_+}\nonumber\right.\\
&\left.\wedge e^{\za x}\breve{\Phi}(x,k)^T\be_{n+1-m_-}\wedge
\ldots\wedge e^{\za x}\breve{\Phi}(x,k)^T\be_n\right),\label{3.12b}
\end{align}
\end{subequations}
are solutions of the AKNS system \eqref{1.1}. Using that
\begin{align*}
\breve{\Psi}(x,k)^T\be_j&=\begin{cases}e^{-ikx}\breve{M}(x,k)^T\be_j,
&j=1,\ldots,m_+,\\e^{ikx}\breve{M}(x,k)^T\be_j,&j=n+1-m_-,\ldots,n,\end{cases}\\
\breve{\Phi}(x,k)^T\be_j&=\begin{cases}e^{-ikx}\breve{N}(x,k)^T\be_j,
&j=1,\ldots,m_+,\\e^{ikx}\breve{N}(x,k)^T\be_j,&j=n+1-m_-,\ldots,n,\end{cases}
\end{align*}
while
$$\left(e^{\za x}e^{-ikx}\right)^{m_+}\left(e^{\za x}e^{ikx}\right)^{m_-}
=\left[e^{-\tfrac{2ikm_-x}{n-1}}\right]^{m_+}\left[e^{\tfrac{2ikm_+x}{n-1}}
\right]^{m_-}=1,$$
we obtain
\begin{subequations}\label{3.13}
\begin{align}
\zO_+(x,k)&\simeq(-1)^{m_+m_-}\left(\breve{N}(x,k)^T\be_1\wedge\ldots\wedge
\breve{N}(x,k)^T\be_{m_+}\nonumber\right.\\&\left.\wedge
\breve{M}(x,k)^T\be_{n+1-m_-}\wedge\ldots\wedge\breve{M}(x,k)^T\be_n\right),
\label{3.13a}\\
\zO_-(x,k)&\simeq(-1)^{m_+m_-}\left(\breve{M}(x,k)^T\be_1\wedge\ldots\wedge
\breve{M}(x,k)^T\be_{m_+}\nonumber\right.\\&\left.\wedge
\breve{N}(x,k)^T\be_{n+1-m_-}\wedge\ldots\wedge\breve{N}(x,k)^T\be_n\right).
\label{3.13b}
\end{align}
\end{subequations}
Hence, $\zO^\pm(x,k)$ is continuous in $k\in\C^\pm\cup\R$, is analytic in
$k\in\C^\pm$, and tends to $(-1)^{m_+m_-}\bg_{m_++1}=\bd_{m_++1}$ (to be
identified with $\be_{m_++1}$) as $k\to\infty$ from within $\C^\pm\cup\R$.

Let us now write
\begin{align*}
\bXi^+(x,k)&=\begin{bmatrix}\Psi(x,k)\be_+&\zO^+(x,k)&\Phi(x,k)\be_-
\end{bmatrix},\\
\bXi^-(x,k)&=\begin{bmatrix}\Phi(x,k)\be_+&\zO^-(x,k)&\Psi(x,k)\be_-
\end{bmatrix}.
\end{align*}
Then $\bXi^\pm(x,k)$ is an $n\times n$ matrix solution of the AKNS system
\eqref{1.1} such that $\bXi^\pm(x,k)e^{-ikx\zS}$ is continuous in
$k\in\C^\pm\cup\R$, is analytic in $k\in\C^\pm$, and tends to $I_n$ as
$k\to\infty$ from within $\C^\pm\cup\R$. We observe that
\begin{align*}
\bXi^+(x,k)e^{-ikx\zS}&=\begin{bmatrix}N(x,k)\be_+&\zO^+(x,k)&M(x,k)\be_-
\end{bmatrix},\\
\bXi^-(x,k)e^{-ikx\zS}&=\begin{bmatrix}M(x,k)\be_+&\zO^-(x,k)&N(x,k)\be_-
\end{bmatrix}.
\end{align*}

\section{Scattering matrices}\label{sec:4}

In this section we define the transition coefficients $A_l(k)$ and $A_r(k)$
which act as matrix coupling constants between the Jost matrices $\Psi(x,k)$ and
$\Phi(x,k)$ and between the dual Jost matrices $\breve{\Psi}(x,k)$ and
$\breve{\Phi}(x,k)$. We then go to define the transition coefficients coupling
$\bXi^+(x,k)$ and $\bXi^-(x,k)$.

\subsection{Transition matrices}\label{sec:4.1}

Since the Jost solutions $\Psi(x,k)$ and $\Phi(x,k)$ satisfy the same
homogeneous first order system \eqref{1.1}, there exist so-called transition
matrices $A_l(k)$ and $A_r(k)$, one the inverse of the other, such that
\begin{equation}\label{4.1}
\Psi(x,k)=\Phi(x,k)A_l(k),\qquad\Phi(x,k)=\Psi(x,k)A_r(k),
\end{equation}
where $(x,k)\in\R^2$. Similarly, since the dual Jost solutions satisfy the same
homogeneous linear first order system \eqref{3.2}, there exist so-called dual
transition matrices $\breve{A}_l(k)$ and $\breve{A}_r(k)$, one the inverse of
the other, such that
\begin{equation}\label{4.2}
\breve{\Psi}(x,k)=\breve{A}_l(k)\breve{\Phi}(x,k),\qquad
\breve{\Phi}(x,k)=\breve{A}_r(k)\breve{\Psi}(x,k),
\end{equation}
where $(x,k)\in\R^2$. Then \eqref{3.4} imply that
\begin{equation}\label{4.3}
\breve{A}_l(k)=A_l(k)^{-1}=A_r(k),\qquad\breve{A}_r(k)=A_r(k)^{-1}=A_l(k),
\end{equation}
where $\det A_l(k)=\det A_r(k)=1$ for every $k\in\R$. For $\bzs$-focusing
potentials commuting with $\zS$ we obtain for $k\in\R$ the symmetry relations
\begin{equation}\label{4.4}
A_l(k)^\dagger=\bzs A_r(k)\bzs,\qquad A_r(k)^\dagger=\bzs A_l(k)\bzs,
\end{equation}
thus making $A_l(k)$ and $A_r(k)$ $\bzs$-unitary. Thus for focusing potentials
the transition matrices $A_l(k)$ and $A_r(k)$ are unitary. For $\bzs$-focusing
potentials anticommuting with $\zS$ we obtain for $k\in\R$ the symmetry
relations
\begin{equation}\label{4.5}
A_l(k)^\dagger=\bzs A_r(-k)\bzs,\qquad A_r(k)^\dagger=\bzs A_l(-k)\bzs.
\end{equation}

Premultiplying \eqref{2.2} by $e^{-ikx\zS}$ and using \eqref{4.1}, we obtain
\begin{subequations}\label{4.6}
\begin{align}
A_l(k)&=I_n-\int_{-\infty}^\infty dy\,e^{-iky\zS}\CQ(y)M(y,k)e^{iky\zS},
\label{4.6a}\\
A_r(k)&=I_n+\int_{-\infty}^\infty dy\,e^{-iky\zS}\CQ(y)N(y,k)e^{iky\zS}.
\label{4.6b}
\end{align}
\end{subequations}
As a result,
\begin{subequations}\label{4.7}
\begin{align}
A_l(k)_{++}&=\be_+^TA_l(k)\be_+=I_{m_+}-\int_{-\infty}^\infty dy\,
\be_+^T\CQ(y)M(y,k)\be_+,\label{4.7a}\\
A_r(k)_{--}&=\be_-^TA_r(k)\be_-=I_{m_-}+\int_{-\infty}^\infty dy\,
\be_-^T\CQ(y)N(y,k)\be_-,\label{4.7b}\\
\intertext{are continuous in $k\in\C^+\cup\R$, are analytic in $k\in\C^+$, and
tend to the identity as $k\to\infty$ from within $\C^+\cup\R$, whereas}
A_r(k)_{++}&=\be_+^TA_r(k)\be_+=I_{m_+}+\int_{-\infty}^\infty dy\,
\be_+^T\CQ(y)N(y,k)\be_+,\label{4.7c}\\
A_l(k)_{--}&=\be_-^TA_l(k)\be_-=I_{m_-}-\int_{-\infty}^\infty dy\,
\be_-^T\CQ(y)N(y,k)\be_-,\label{4.7d}
\end{align}
\end{subequations}
are continuous in $k\in\C^-\cup\R$, are analytic in $k\in\C^-$, and tend to the
identity as $k\to\infty$ from within $\C^-\cup\R$.

\subsection{Asymptotic properties}\label{sec:4.2}

Using \eqref{2.2} and \eqref{3.1} we easily prove that
\begin{alignat*}{3}
e^{-ikx\zS}\Psi(x,k)&\to\begin{cases}I_n,&x\to+\infty,\\A_l(k),&x\to-\infty,
\end{cases}&\ e^{-ikx\zS}\Phi(x,k)&\to\begin{cases}A_r(k),&x\to+\infty,\\
I_n,&x\to-\infty,\end{cases}\\
\breve{\Psi}(x,k)e^{ikx\zS}&\to\begin{cases}I_n,&x\to+\infty,\\
A_r(k),&x\to-\infty,\end{cases}&\
\breve{\Phi}(x,k)e^{ikx\zS}&\to\begin{cases}A_l(k),&x\to+\infty,\\
I_n,&x\to-\infty.\end{cases}
\end{alignat*}
Using the above, we easily verify the asymptotic identities
\begin{subequations}\label{4.8}
\begin{align}
e^{-ikx\zS}\bXi^+(x,k)&\to
\begin{cases}\ClD_+^\up(k)=\begin{bmatrix}I_{m_+}&\ClD_+^\up(k)_{+0}
&A_r(k)_{+-}\\ \bzero&\ClD_+^\up(k)_{00}&A_r(k)_{0-}\\
\bzero&\xcancel{\ClD_+^\up(k)_{-0}}&A_r(k)_{--}\end{bmatrix},
&x\to+\infty,\\
\ClD_+^\dn(k)=\begin{bmatrix}A_l(k)_{++}&\xcancel{\ClD_+^\dn(k)_{+0}}
&\bzero\\A_l(k)_{0+}&\ClD_+^\dn(k)_{00}&\bzero\\
A_l(k)_{-+}&\ClD_+^\dn(k)_{-0}&I_{m_-}\end{bmatrix},&x\to-\infty,
\end{cases}\label{4.8a}\\
e^{-ikx\zS}\bXi^-(x,k)&\to\begin{cases}
\ClD_-^\dn(k)=\begin{bmatrix}A_r(k)_{++}&\xcancel{\ClD_-^\dn(k)_{+0}}
&\bzero\\A_r(k)_{0+}&\ClD_-^\dn(k)_{00}&\bzero\\
A_r(k)_{-+}&\ClD_-^\dn(k)_{-0}&I_{m_-}\end{bmatrix},&x\to+\infty,\\
\ClD_-^\up(k)=\begin{bmatrix}I_{m_+}&\ClD_-^\up(k)_{+0}&A_l(k)_{+-}\\
\bzero&\ClD_-^\up(k)_{00}&A_l(k)_{0-}\\
\bzero&\xcancel{\ClD_-^\up(k)_{-0}}&A_l(k)_{--}\end{bmatrix},
&x\to-\infty,\end{cases}\label{4.8b}
\end{align}
\end{subequations}
where the cancelled entries will turn out to be zero matrices. Then the second
columns coincide with the expressions
\begin{subequations}\label{4.9}
\begin{align}
&\ClD_+^\up(k)\be_0=(-1)^{m_+m_-}\times\label{4.9a}\\&\times\left(
e^{ikx\bN_-}A_l(k)^T\be_1\wedge\ldots\wedge e^{ikx\bN_-}A_l(k)^T\be_{m_+}\wedge
\be_{n+1-m_+}\wedge\ldots\wedge\be_n\right),\nonumber\\
&\ClD_+^\dn(k)\be_0=(-1)^{m_+m_-}\times\label{4.9b}\\&\times\left(\be_1\wedge
\ldots\wedge\be_{m_+}\wedge e^{-ikx\bN_+}A_r(k)^T\be_{n+1-m_-}\wedge\ldots\wedge
e^{-ikx\bN_+}A_r(k)^T\be_n\right),\nonumber\\
&\ClD_-^\dn(k)\be_0=(-1)^{m_+m_-}\times\label{4.9c}\\&\times\left(\be_1\wedge
\ldots\wedge\be_{m_+}\wedge e^{-ikx\bN_+}A_l(k)^T\be_{n+1-m_+}\wedge\ldots\wedge
e^{-ikx\bN_+}A_l(k)^T\be_n\right),\nonumber\\
&\ClD_-^\up(k)\be_0=(-1)^{m_+m_-}\times\label{4.9d}\\&\times\left(
e^{ikx\bN_-}A_r(k)^T\be_1\wedge\ldots\wedge e^{ikx\bN_-}A_r(k)^T\be_{m_+}\wedge
\be_{n+1-m_+}\wedge\ldots\wedge\be_n\right),\nonumber
\end{align}
\end{subequations}
where $\bN_+=2I_{m_+}\oplus 1\oplus 0_{m_-}$ and
$\bN_-=0_{m_+}\oplus 1\oplus2I_{m_-}$. Then the right-hand sides of \eqref{4.9a}
and \eqref{4.9d} are linear combinations of $\bd_1,\ldots,\bd_{m_++1}$, whereas
the right-hand sides of \eqref{4.9b} and \eqref{4.9c} are linear combinations of
$\bd_{m_++1},\ldots,\bd_n$. Hence,
\begin{subequations}\label{4.10}
\begin{align}
\ClD_+^\up(k)_{-0}&=\ClD_-^\up(k)_{-0}=0_{m_-\times1},\label{4.10a}\\
\ClD_+^\dn(k)_{+0}&=\ClD_-^\dn(k)_{+0}=0_{m_+\times1},\label{4.10b}
\end{align}
\end{subequations}
thus justifying the cancellations of certain entries of \eqref{4.8a} and
\eqref{4.8b}. For later use we now evaluate the inverses
\begin{subequations}\label{4.11}
\begin{align}
\ClD_+^\up(k)^{-1}&=\begin{bmatrix}I_{m_+}&-\ClD_+^\up(k)_{+0}
&\ClD_+^\up(k)_{+0}\ClD_+^\up(k)_{00}^{-1}A_r(k)_{0-}-A_r(k)_{+-}\\
\bzero&1&-\ClD_+^\up(k)_{00}^{-1}A_r(k)_{0-}\\ \bzero&\bzero
&I_{m_-}\end{bmatrix}\times\nonumber\\
&\times\text{diag}\left(I_{m_+},\ClD_+^\up(k)_{00}^{-1},A_r(k)_{--}^{-1}
\right),\label{4.11a}\\
\ClD_-^\dn(k)^{-1}&=\begin{bmatrix}I_{m_+}&\bzero&\bzero\\
-\ClD_-^\dn(k)_{00}^{-1}A_r(k)_{0+}&1&\bzero\\
\ClD_-^\dn(k)_{-0}\ClD_-^\dn(k)_{00}^{-1}A_r(k)_{0+}-A_r(k)_{-+}
&-\ClD_-^\dn(k)_{-0}&I_{m_-}\end{bmatrix}\times\nonumber\\
&\times\text{diag}\left(A_r(k)_{++}^{-1},\ClD_-^\dn(k)_{00}^{-1},I_{m_-}
\right),\label{4.11b}\\
\ClD_-^\up(k)^{-1}&=\begin{bmatrix}I_{m_+}&-\ClD_-^\up(k)_{+0}
&\ClD_-^\up(k)_{+0}\ClD_-^\up(k)_{00}^{-1}A_l(k)_{0-}-A_l(k)_{+-}\\
\bzero&1&-\ClD_-^\up(k)_{00}^{-1}A_l(k)_{0-}\\ \bzero&\bzero
&I_{m_-}\end{bmatrix}\times\nonumber\\
&\times\text{diag}\left(I_{m_+},\ClD_-^\up(k)_{00}^{-1},A_l(k)_{--}^{-1}
\right),\label{4.11c}\\
\ClD_+^\dn(k)^{-1}&=\begin{bmatrix}I_{m_+}&\bzero&\bzero\\
-\ClD_+^\dn(k)_{00}^{-1}A_l(k)_{0+}&1&\bzero\\
\ClD_+^\dn(k)_{-0}\ClD_+^\dn(k)_{00}^{-1}A_l(k)_{0+}-A_l(k)_{-+}
&-\ClD_+^\dn(k)_{-0}&I_{m_-}\end{bmatrix}\times\nonumber\\
&\times\text{diag}\left(A_l(k)_{++}^{-1},\ClD_+^\dn(k)_{00}^{-1},I_{m_-}
\right).\label{4.11d}
\end{align}
\end{subequations}

Since $\bXi^+(x,k)$ and $\bXi^-(x,k)$ are $n\times n$ matrix solutions of the
same homogeneous linear first order system, there exist {\it scattering
matrices} $S(k)$ and $\tS(k)$, one the inverse of the other, such that
\begin{equation}\label{4.12}
\bXi^-(x,k)=\bXi^+(x,k)S(k),\qquad\bXi^+(x,k)=\bXi^-(x,k)\tS(k),
\end{equation}
where $(x,k)\in\R^2$. We call the block diagonal entries of $S(k)$ and $\tS(k)$
{\it transmission coefficients} and minus their block off-diagonal entries
{\it reflection coefficients}. Using the asymptotic conditions \eqref{4.8}
we also get
\begin{subequations}\label{4.13}
\begin{align}
S(k)=\ClD_+^\up(k)^{-1}\ClD_-^\dn(k)=\ClD_+^\dn(k)^{-1}\ClD_-^\up(k),
\label{4.13a}\\
\tS(k)=\ClD_-^\dn(k)^{-1}\ClD_+^\up(k)=\ClD_-^\up(k)^{-1}\ClD_+^\dn(k).
\label{4.13b}
\end{align}
\end{subequations}
Using \eqref{4.8} and \eqref{4.11} we get from \eqref{4.13} the four expressions
\begin{subequations}\label{4.14}
\begin{align}
S(k)&=\begin{bmatrix}I_{m_+}&-\ClD_+^\up(k)_{+0}
&\ClD_+^\up(k)_{+0}\ClD_+^\up(k)_{00}^{-1}A_r(k)_{0-}-A_r(k)_{+-}\\
\bzero&1&-\ClD_+^\up(k)_{00}^{-1}A_r(k)_{0-}\\ \bzero&\bzero
&I_{m_-}\end{bmatrix}\times\nonumber\\ &\times
\begin{bmatrix}A_r(k)_{++}&\bzero&\bzero\\ \ClD_+^\up(k)_{00}^{-1}A_r(k)_{0+}
&\ClD_+^\up(k)_{00}^{-1}\ClD_-^\dn(k)_{00}&\bzero\\ A_r(k)_{--}^{-1}A_r(k)_{-+}
&A_r(k)_{--}^{-1}\ClD_-^\dn(k)_{-0}&A_r(k)_{--}^{-1}\end{bmatrix},
\label{4.14a}
\end{align}
\begin{align}
\tS(k)&=\begin{bmatrix}I_{m_+}&\bzero&\bzero\\
-\ClD_-^\dn(k)_{00}^{-1}A_r(k)_{0+}&1&\bzero\\
\ClD_-^\dn(k)_{-0}\ClD_-^\dn(k)_{00}^{-1}A_r(k)_{0+}-A_r(k)_{-+}
&-\ClD_-^\dn(k)_{-0}&I_{m_-}\end{bmatrix}\times\nonumber\\
&\times\begin{bmatrix}A_r(k)_{++}^{-1}&A_r(k)_{++}^{-1}\ClD_+^\up(k)_{+0}
&A_r(k)_{++}^{-1}A_r(k)_{+-}\\ \bzero&\ClD_-^\dn(k)_{00}^{-1}\ClD_+^\up(k)_{00}
&\ClD_-^\dn(k)_{00}^{-1}A_r(k)_{0-}\\ \bzero&\bzero&A_r(k)_{--}\end{bmatrix},
\label{4.14b}\\
S(k)&=\begin{bmatrix}I_{m_+}&\bzero&\bzero\\
-\ClD_+^\dn(k)_{00}^{-1}A_l(k)_{0+}&1&\bzero\\
\ClD_+^\dn(k)_{-0}\ClD_+^\dn(k)_{00}^{-1}A_l(k)_{0+}-A_l(k)_{-+}
&-\ClD_+^\dn(k)_{-0}&I_{m_-}\end{bmatrix}\times\nonumber\\
&\times\begin{bmatrix}A_l(k)_{++}^{-1}&A_l(k)_{++}^{-1}\ClD_-^\up(k)_{+0}
&A_l(k)_{++}^{-1}A_l(k)_{+-}\\ \bzero&\ClD_+^\dn(k)_{00}^{-1}\ClD_-^\up(k)_{00}
&\ClD_+^\dn(k)_{00}^{-1}A_l(k)_{0-}\\ \bzero&\bzero&A_l(k)_{--}\end{bmatrix},
\label{4.14c}\\
\tS(k)&=\begin{bmatrix}I_{m_+}&-\ClD_-^\up(k)_{+0}
&\ClD_-^\up(k)_{+0}\ClD_-^\up(k)_{00}^{-1}A_l(k)_{0-}-A_l(k)_{+-}\\
\bzero&1&-\ClD_-^\up(k)_{00}^{-1}A_l(k)_{0-}\\ \bzero&\bzero
&I_{m_-}\end{bmatrix}\times\nonumber\\
&\times\begin{bmatrix}A_l(k)_{++}&\bzero&\bzero\\
\ClD_-^\up(k)_{00}^{-1}A_l(k)_{0+}&\ClD_-^\up(k)_{00}^{-1}\ClD_+^\dn(k)_{00}
&\bzero\\ A_l(k)_{--}^{-1}A_l(k)_{-+}&A_l(k)_{--}^{-1}\ClD_+^\dn(k)_{-0}
&A_l(k)_{--}^{-1}\end{bmatrix}.\label{4.14d}
\end{align}
\end{subequations}
Choosing the simpler expression for the block entries from the two versions of
$S(k)$ and the two versions of $\tS(k)$, we obtain
\begin{subequations}\label{4.15}
\begin{align}
S(k)&=\begin{bmatrix}A_l(k)_{++}^{-1}&A_l(k)_{++}^{-1}
\ClD_-^\up(k)_{+0}&A_l(k)_{++}^{-1}A_l(k)_{+-}\\S(k)_{0+}
&S(k)_{00}&S(k)_{0-}\\
A_r(k)_{--}^{-1}A_r(k)_{-+}&A_r(k)_{--}^{-1}\ClD_-^\dn(k)_{-0}
&A_r(k)_{--}^{-1}\end{bmatrix},\label{4.15a}\\
\tS(k)&=\begin{bmatrix}A_r(k)_{++}^{-1}&A_r(k)_{++}^{-1}
\ClD_+^\up(k)_{+0}&A_r(k)_{++}^{-1}A_r(k)_{+-}\\ \tS(k)_{0+}
&\tS(k)_{00}&\tS(k)_{0-}\\
A_l(k)_{--}^{-1}A_l(k)_{-+}&A_l(k)_{--}^{-1}\ClD_+^\dn(k)_{-0}
&A_l(k)_{--}^{-1}\end{bmatrix},\label{4.15b}
\end{align}
\end{subequations}
where
\begin{align*}
S(k)_{0+}&=-\ClD_+^\dn(k)_{00}^{-1}A_l(k)_{0+}A_l(k)_{++}^{-1},\\
S(k)_{0-}&=-\ClD_+^\up(k)_{00}^{-1}A_r(k)_{0-}A_r(k)_{--}^{-1},\\
\tS(k)_{0+}&=-\ClD_-^\dn(k)_{00}^{-1}A_r(k)_{0+}A_r(k)_{++}^{-1},\\
\tS(k)_{0-}&=-\ClD_-^\up(k)_{00}^{-1}A_l(k)_{0-}A_l(k)_{--}^{-1},\\
S(k)_{00}&=\ClD_+^\up(k)_{00}^{-1}
\left(\ClD_-^\dn(k)_{00}-A_r(k)_{0-}
A_r(k)_{--}^{-1}\ClD_-^\dn(k)_{-0}\right)\\
&=\ClD_+^\dn(k)_{00}^{-1}\left(\ClD_-^\up(k)_{00}-A_l(k)_{0+}
A_l(k)_{++}^{-1}\ClD_-^\up(k)_{+0}\right),\\
\tS(k)_{00}&=\ClD_-^\up(k)_{00}^{-1}\left(\ClD_+^\dn(k)_{00}
-A_l(k)_{0-}A_l(k)_{--}^{-1}\ClD_+^\dn(k)_{-0}\right)\\
&=\ClD_-^\dn(k)_{00}^{-1}\left(\ClD_+^\up(k)_{00}-A_r(k)_{0+}
A_r(k)_{++}^{-1}\ClD_+^\up(k)_{+0}\right).
\end{align*}

By a {\it spectral singularity} we mean a real zero of any of the functions
$\det A_l(k)_{++}$, $\det A_r(k)_{++}$, $\det A_l(k)_{--}$, $\det A_r(k)_{--}$,
$S(k)_{00}$, and $\tS(k)_{00}$. Assuming the absence of spectral singularities,
we denote by $w$ and $\tw$ the winding numbers of the curves in the complex
plane obtained by computing the values of the respective functions $S(k)_{00}$
and $\tS(k)_{00}$ as $k$ runs along the real line from $-\infty$ to $+\infty$.
As a result, there exist unique functions $\A^\pm(k)$ and $\tilde{\A}^\pm(k)$
that are continuous in $k\in\C^\pm\cup\R$, are analytic in $k\in\C^\pm$, do not
have any zeros $k\in\C^\pm\cup\R$, and tend to $1$ as $k\to\infty$ from within
$\C^\pm\cup\R$ such that the following Wiener-Hopf factorizations hold true
\cite{GF,Kr}:
\begin{subequations}\label{4.16}
\begin{align}
S(k)_{00}&=\left(\frac{k-i}{k+i}\right)^w\A^+(k)\A^-(k),\label{4.16a}\\
\tS(k)_{00}&=\left(\frac{k-i}{k+i}\right)^{\tw}\tilde{\A}^+(k)\tilde{\A}^-(k).
\label{4.16b}
\end{align}
\end{subequations}
Moreover, for certain $L^1$-functions $\K^\pm$ and $\tilde{\K}^\pm$ we have
$$\A^\pm(k)=1\pm\int_0^{\pm\infty}dy\,e^{iky}\K^\pm(y),\qquad
\tilde{\A}^\pm(k)=1\pm\int_0^{\pm\infty}dy\,e^{iky}\tilde{\K}^\pm(y).$$

Under the condition that $w=\tw=0$, let us write the scattering matrices as the
sums of a block diagonal matrix and a block off-diagonal matrix by means of
correction factors as follows:
\begin{subequations}\label{4.17}
\begin{align}
S(k)&=\left(\CT(k)-\tilde{\CC}(k)^{-1}\CR(k)\right)\CC(k),\label{4.17a}\\
\tS(k)&=\left(\tilde{\CT}(k)-\CC(k)^{-1}\tilde{\CR}(k)\right)\tilde{\CC}(k),
\label{4.17b}
\end{align}
\end{subequations}
where
\begin{align*}
\CT(k)&=\text{diag}\left(A_l(k)_{++}^{-1},\A_l^+(k),A_r(k)_{--}^{-1}\right),\\
\tilde{\CT}(k)&=\text{diag}\left(A_r(k)_{++}^{-1},\A_r^-(k),A_l(k)_{--}^{-1}
\right),
\end{align*}
contain the so-called {\it corrected transmission coefficients} and
\begin{align*}
\CC(k)&=\text{diag}\left(I_{m_+},\A_l^-(k),I_{m_-}\right),\\
\tilde{\CC}(k)&=\text{diag}\left(I_{m_+},\A_r^+(k),I_{m_-}\right),
\end{align*}
are the so-called {\it correction factors}. We call the off-diagonal matrices
$\CR(k)$ and $\tilde{\CR}(k)$ the {\it corrected reflection coefficients}. We
then arrive at the following Riemann-Hilbert problems (please note that, as
usual, the Riemann-Hilbert problems are formulated along the real axis in the
complex $k$-plane)
\begin{subequations}\label{4.18}
\begin{align}
\bXi^-(x,k)\CC(k)^{-1}&=\bXi^+(x,k)\CT(k)-\bXi^+(x,k)\tilde{\CC}(k)^{-1}\CR(k),
\label{4.18a}\\
\bXi^+(x,k)\tilde{\CC}(k)^{-1}&=\bXi^-(x,k)\tilde{\CT}(k)-\bXi^-(x,k)\CC(k)^{-1}
\tilde{\CR}(k),\label{4.18b}
\end{align}
\end{subequations}
where
\begin{subequations}\label{4.19}
\begin{align}
\bXi^+(x,k)\tilde{\CC}(k)^{-1}&=\left[I_n+\int_0^\infty d\zg\,
e^{ik\zg}\K^+(x,x+\zg)\right]e^{ikx\zS},\label{4.19a}\\
\bXi^-(x,k)\CC(k)^{-1}&=\left[I_n+\int_0^\infty d\zg\,
e^{-ik\zg}\K^-(x,x-\zg)\right]e^{ikx\zS},\label{4.19b}
\end{align}
\end{subequations}
and
$$\int_0^\infty d\zg\left(\|\K^+(x,x+\zg)\|+\|\K^-(x,x-\zg)\|\right)<+\infty.$$

To take into account the so-called norming constants when formulating the
Riemann-Hilbert problems \eqref{4.18}, we assume that the zeros of $A_r(k)_{++}$
and $A_l(k)_{--}$ in $\C^-$ are simple and nonoverlapping and that the winding
numbers $w$ and $\tw$ both vanish. Then the zeros $k_j$ of
$A_r(k)_{--}A_l(k)_{++}$ in $\C^+$ are simple and the zeros of $\tk_j$ of
$A_l(k)_{--}A_r(k)_{++}$ in $\C^+$ are simple. Moreover, the absence of spectral
singularities implies that the determinants of $A_r(k)_{++}$, $A_l(k)_{++}$,
$A_r(k)_{--}$, and $A_l(k)_{--}$ do not have any real zeros. As a result, the
reduced transmission coefficients $\CT(k)$ and $\tilde{\CT}(k)$ are continuous
in $k\in\R$, while $\tilde{\CT}(k)$ is meromorphic in $k\in\C^+$ with simple
zeros $k_j$ and $\CT(k)$ is meromorphic in $k\in\C^-$ with simple poles $\tk_j$.
Letting $\btau_j$ stand for the residue of $\tilde{\CT}(k)$ at $k=k_j$ and
$\tilde{\btau}_j$ for the residue of $\CT(k)$ at $k=\tk_j$, we define the
{\it norming constants} as those off-block-diagonal $n\times n$ matrices $\bN_j$
and $\tilde{\bN}_j$ such that
\begin{subequations}\label{4.20}
\begin{align}
\bXi^+(x,k_j)\btau_j&=+i\bXi^+(x,k_j)\tilde{\CC}(k_j)^{-1}\bN_j,\label{4.20a}\\
\bXi^-(x,\tk_j)\tilde{\btau}_j&=-i\bXi^-(x,\tk_j)\CC(\tk_j)^{-1}\tilde{\bN}_j.
\label{4.20b}
\end{align}
\end{subequations}
We can then write the Riemann-Hilbert problems \eqref{4.18} in their final form
\begin{subequations}\label{4.21}
\begin{align}
 F^-(x,k)&\CC(k)^{-1}=F^+(x,k)+\sum_j\frac{F^+(x,k)-F^+(x,k_j}{k-k_j}\btau_j
\label{4.21a}\\&+i\sum_j\frac{F^+(x,k_j)\tilde{\CC}(k_j)^{-1}\bN_j}{k-k_j}
-F^+(x,k)\tilde{\CC}(k)^{-1}e^{ikx\zS}\CR(k)e^{-ikx\zS},\nonumber\\
 F^+(x,k)&\tilde{\CC}(k)^{-1}=F^-(x,k)+\sum_j\frac{F^-(x,k)-F^-(x,\tk_j)}
{k-\tk_j}\tilde{\btau}_j\label{4.21b}\\
&-i\sum_j\frac{F^-(x,\tk_j)\CC(\tk_j)^{-1}\tilde{\bN}_j}{k-\tk_j}
-F^-(x,k)\CC(k)^{-1}e^{ikx\zS}\tilde{\CR}(k)e^{-ikx\zS}.\nonumber
\end{align}
\end{subequations}

\section{Marchenko integral equations}\label{sec:5}

In this section we convert the Riemann-Hilbert problems \eqref{4.21} into two
coupled systems of Marchenko integral equations, one satisfied by $\K^+(x,y)$
and the other satisfied by $\K^-(x,y)$. Here we assume that the zeros of
$A_l(k)_{--}$  and $A_r(k)_{++}$ in $\C^+$ are simple and nonoverlapping, that
the zeros of $A_r(k)_{--}$  and $A_l(k)_{++}$ in $\C^-$ are simple and
nonoverlapping, that these four functions as well as $S(k)_{00}$ and
$\tS(k)_{00}$ do not have any real zeros, and that the winding numbers $w$ of
$S(k)_{00}$ and $\tw$ of $\tS(k)_{00}$ both vanish.

We now observe that the functions $A_r(k)_{--}$ and $A_l(k)_{++}$ have their
entries in $\CW^+$, the functions $A_l(k)_{--}$ and $A_r(k)_{++}$ have their
entries in $\CW^-$, and the entries of the functions $F^\pm(x,k)$ belong to
$\CW^\pm$ for each $x\in\R$. Here $\CW^+$ and $\CW^-$ are subalgebras of the
Wiener algebra $\CW$ (see App.~\ref{sec:B}). Letting $\Pi^\pm$ be the
projections defined by \eqref{B.1}, we apply $\Pi^-$ to \eqref{4.21a} and
$\Pi^+$ to \eqref{4.21b} and obtain
\begin{subequations}\label{5.1}
\begin{align}
 F^-(x,k)\CC(k)^{-1}-I_n&=+i\sum_j\frac{F^+(x,k_j)\tilde{\CC}(k_j)^{-1}\bN_j}
{k-k_j}\nonumber\\&-\Pi^-\left(F^+(x,k)\tilde{\CC}(k)^{-1}e^{ikx\zS}\CR(k)
e^{-ikx\zS}\right),\label{5.1a}\\
 F^+(x,k)\tilde{\CC}(k)^{-1}-I_n&=-i\sum_j\frac{F^-(x,\tk_j)\CC(\tk_j)^{-1}
\tilde{\bN}_j}{k-\tk_j}\nonumber\\&-\Pi^+\left(F^-(x,k)\CC(k)^{-1}e^{ikx\zS}
\tilde{\CR}(k)e^{-ikx\zS}\right).\label{5.1b}
\end{align}
\end{subequations}

We now write \eqref{5.1a} as an equation in $\CW^-_0$ (definition of $\CW^-_0$
is given in App.~\ref{sec:B}) and \eqref{5.1b} as an equation in $\CW^+_0$,
strip off the Fourier transform, and end up with a coupled set of two equations
for $\K^\pm(x,x\pm\zg)$. To do so, we apply \eqref{4.19}, the identities
\begin{alignat*}{3}
\frac{-i}{k-k_j}&=\int_0^\infty ds\,e^{-i(k-k_j)s},&\qquad&k\in\C^-,\\
\frac{+i}{k-\tk_j}&=\int_0^\infty ds\,e^{i(k-\tk_j)s},&\qquad&k\in\C^+,
\end{alignat*}
and the equations
\begin{subequations}\label{5.2}
\begin{align}
e^{ikx\zS}\CR(k)e^{-ikx\zS}&=\begin{bmatrix}0&e^{ikx}\CR(k)_{12}
&e^{2ikx}\CR(k)_{13}\\e^{-ikx}\CR(k)_{21}&0&e^{ikx}\CR(k)_{23}\\
e^{-2ikx}\CR(k)_{31}&e^{-ikx}\CR(k)_{32}&0\end{bmatrix}\nonumber
\end{align}
\begin{align}
&=\int_{-\infty}^\infty d\zg\,e^{-ik\zg}\begin{bmatrix}0&\hat{\CR}(\zg+x)_{12}
&\hat{\CR}(\zg+2x)_{13}\\ \hat{\CR}(\zg-x)_{21}&0&\hat{\CR}(\zg+x)_{23}\\
\hat{\CR}(\zg-2x)_{31}&\hat{\CR}(\zg-x)_{32}&0\end{bmatrix}\nonumber\\
&=\int_{-\infty}^\infty d\zg\,e^{-ik\zg}
\left[\hat{\CR}(\zg+(\zS_i-\zS_j)x)_{ij}\right]_{i,j=1}^3,\label{5.2a}\\
e^{ikx\zS}\tilde{\CR}(k)e^{-ikx\zS}&=\begin{bmatrix}0&e^{ikx}\tilde{\CR}(k)_{12}
&e^{2ikx}\tilde{\CR}(k)_{13}\\e^{-ikx}\tilde{\CR}(k)_{21}&0
&e^{ikx}\tilde{\CR}(k)_{23}\\e^{-2ikx}\tilde{\CR}(k)_{31}
&e^{-ikx}\tilde{\CR}(k)_{32}&0\end{bmatrix}\nonumber\\
&=\int_{-\infty}^\infty d\zg\,e^{-ik\zg}\begin{bmatrix}
0&\hat{\tilde{\CR}}(\zg-x)_{12}&\hat{\tilde{\CR}}(\zg-2x)_{13}\\
\hat{\tilde{\CR}}(\zg+x)_{21}&0&\hat{\tilde{\CR}}(\zg-x)_{23}\\
\hat{\tilde{\CR}}(\zg+2x)_{31}&\hat{\tilde{\CR}}(\zg+x)_{32}&0\end{bmatrix}
\nonumber\\
&=\int_{-\infty}^\infty d\zg\,e^{ik\zg}
\left[\hat{\CR}(\zg+(\zS_j-\zS_i)x)_{ij}\right]_{i,j=1}^3,\label{5.2b}
\end{align}
\end{subequations}
where $\zS=\text{diag}(\zS_1I_{m_+},\zS_2,\zS_3I_{m_-})$ with $\zS_1=1$,
$\zS_2=0$, and $\zS_3=-1$, and the entries of $\hat{\CR}$ and
$\hat{\tilde{\CR}}$ belong to $L^1(\R)$. Consequently, we have arrived at the
Marchenko integral equations
\begin{subequations}\label{5.3}
\begin{align}
\K^-(x,x-\zg)&=-\sum_j\left(e^{ik_j\zg}\bN_j+\int_0^\infty d\ozg\,
e^{ik_j(\zg+\ozg)}\K^+(x,x+\ozg)\bN_j\right)\nonumber\\
&-\left[\hat{\CR}(\zg+(\zS_i-\zS_j)x)_{ij}\right]_{i,j=1}^3\nonumber\\
&-\int_0^\infty d\ozg\,\K^+(x,x+\ozg)
\left[\hat{\CR}(\ozg+\zg+(\zS_i-\zS_j)x)_{ij}\right]_{i,j=1}^3,\label{5.3a}\\
\K^+(x,x+\zg)&=-\sum_j\left(e^{-i\tk_j\zg}\tilde{\bN}_j+\int_0^\infty d\ozg\,
e^{-i\tk_j(\zg+\ozg)}\K^-(x,x-\ozg)\tilde{\bN}_j\right)\nonumber\\
&-\left[\hat{\tilde{\CR}}(\zg+(\zS_i-\zS_j)x)_{ij}\right]_{i,j=1}^3\nonumber\\
&-\int_0^\infty d\ozg\,\K^-(x,x-\ozg)
\left[\hat{\tilde{\CR}}(\ozg+\zg+(\zS_j-\zS_i)x)_{ij}\right]_{i,j=1}^3.
\label{5.3b}
\end{align}
\end{subequations}

By using \eqref{5.2b} it is immediate to observe that \eqref{5.3a} and
\eqref{5.3b} can be written as the respective equations
\begin{subequations}\label{5.4}
\begin{align}
\K^-(x,x-\zg)+\zo^+(\zg;x)+\int_0^\infty d\ozg\,\K^+(x,x+\ozg)\zo^+(\ozg+\zg;x)
&=0.\label{5.4a}\\
\K^+(x,x+\zg)+\zo^-(\zg;x)+\int_0^\infty d\ozg\,\K^-(x,x-\ozg)\zo^-(\ozg+\zg;x)
&=0,\label{5.4b}
\end{align}
\end{subequations}
where the Marchenko integral kernels
\begin{subequations}\label{5.5}
\begin{align}
\zo^+(\zg;x)&=\sum_j\,e^{ik_j\zg}\bN_j
+\left[\hat{\CR}(\zg+[\zS_i-\zS_j]x)_{ij}\right]_{i,j=1}^3,\label{5.5a}\\
\zo^-(\zg;x)&=\sum_j\,e^{-i\tk_j\zg}\tilde{\bN}_j
+\left[\hat{\tilde{\CR}}(\zg+[\zS_j-\zS_i]x)_{ij}\right]_{i,j=1}^3,\label{5.5b}
\end{align}
\end{subequations}
are off-block diagonal matrices whose entries belong to $L^1(\R^+)$.

Using the definitions of $F^\pm(x,k)$ we obtain
$$\begin{cases}F^+(x,k)\tilde{\CC}(k)^{-1}\be_+=M(x,k)\be_+,\\
 F^+(x,k)\tilde{\CC}(k)^{-1}\be_-=N(x,k)\be_-,\\
 F^-(x,k)\CC(k)^{-1}\be_+=N(x,k)\be_-,\\
 F^-(x,k)\CC(k)^{-1}\be_-=M(x,k)\be_-.\end{cases}$$
Consequently,
$$\begin{cases}\K^+(x,x)\be_+=K(x,x)\be_+,\quad
\K^+(x,x)\be_-=J(x,x)\be_-,\\
\K^-(x,x)\be_+=J(x,x)\be_+,\quad
\K^-(x,x)\be_-=M(x,x)\be_-,\end{cases}$$
can be expressed in the entries of the potential $\CQ(x)$ using \eqref{2.10} and
\eqref{2.11}.

Let us summarize how to evaluate the potential entries from the solutions
of the Marchenko equations \eqref{5.4}. Indeed, let us first construct the
Marchenko equations \eqref{5.4} from the norming constants and reflection
coefficients by using \eqref{5.5}. Next, we obtain $K(x,x)\be_i$ and
$J(x,x)\be_i$ ($i=1,3$) as the first and third columns of the Marchenko
solutions $\K^\pm(x,y)$ for $y=x$. Finally, we compute the potential entries
from $K(x,x)\be_i$ and $J(x,x)\be_i$ ($i=1,3$) by using \eqref{2.10} and
\eqref{2.11}.

\section{Time evolution of the scattering data}\label{sec:6}

In this section we study the time evolution of the scattering data of the AKNS
system \eqref{1.1} with $\zS=\text{diag}(I_{m_+},0,-I_{m_-})$ associated with
the solution of the generalized long-wave-short-wave equations \eqref{1.3} by
means of the inverse scattering transform.

Let $\bV(x,t)$ be a nonsingular $n\times n$ matrix solution of the first order
system
\begin{equation}\label{6.1}
\bV_t=\bX\bV,\qquad\bV_t=\bT\bV,
\end{equation}
where $(\bX,\bT)$ is the Lax pair given by $\bX=ik\zS+\CQ$ and
$\bT=(ik)^2A-ikB+C$ with
\begin{subequations}\label{6.2}
\begin{align}
A&=\frac{i}{n}\text{diag}(I_{m_+},1-n,I_{m_-}),\label{6.2a}\\
\CQ&=\begin{bmatrix}0_{m_+\times m_+}&S&iL\\T^\dagger&0&S^\dagger\\iK&T
&0_{m_-\times m_-}\end{bmatrix},\label{6.2b}\\
B&=\begin{bmatrix}0_{m_+\times m_+}&iS&0_{m_+\times m_-}\\iT^\dagger&0
&-iS^\dagger\\0_{m_-\times m_+}&-iT&0_{m_-\times m_-}\end{bmatrix},
\label{6.2c}\\
C&=\begin{bmatrix}-iST^\dagger&iS_x-LT&iSS^\dagger\\-iT_x^\dagger-S^\dagger K
&i[T^\dagger S+S^\dagger T]&-iS_x^\dagger-T^\dagger L\\iII^\dagger&iT_x-KS
&-iTS^\dagger\end{bmatrix}.\label{6.2d}
\end{align}
\end{subequations}
Then $\bV(x,t)$ can be postmultiplied in \eqref{6.1} by any nonsingular
$n\times n$ matrix independent of $(x,t)\in\R^2$. For any invertible matrix
solution $W(x,t)$ of \eqref{1.1} there exists an invertible $n\times n$ matrix
$C_W$, not depending on $x$, such that
$$W=\bV C_W^{-1}.$$
Then
\begin{align*}
W_t&=\bV_tC_W^{-1}-\bV C_W^{-1}[C_W]_tC_W^{-1}=\bT\bV C_W^{-1}-\bV C_W^{-1}
[C_W]_tC_W^{-1}\\&=\bT W-W[C_W]_tC_W^{-1}
\end{align*}
implies the important identity
\begin{equation}\label{6.3}
[C_W]_tC_W^{-1}=W^{-1}\bT W-W^{-1}W_t,
\end{equation}
where the left-hand side (and hence also the right-hand side) does not depend on
$x$. We may therefore exploit the asymptotic behavior of $W$ and $\bT$ as
$x\to\pm\infty$ to write the right-hand side of \eqref{6.2} in a simplified way.
Using that $\zS$ and $A$ commute, we thus get as $x\to\pm\infty$
\begin{subequations}\label{6.4}
\begin{align}
[C_\Psi]_tC_\Psi^{-1}&=e^{-ikx\zS}[I_n+0(1)][(ik)^2A+o(1)]e^{ikx\zS}[I_n+o(1)]
=-k^2A,\label{6.4a}\\
[C_\Phi]_tC_\Phi^{-1}&=e^{-ikx\zS}[I_n+0(1)][(ik)^2A+o(1)]e^{ikx\zS}[I_n+o(1)]
=-k^2A.\label{6.4b}
\end{align}
\end{subequations} 
Hence, $C_\Psi(k;t)=e^{-k^2tA}C_\Psi(k;0)$ and
$C_\Phi(k;t)=e^{-k^2tA}C_\Phi(k;0)$.

Recalling the definitions \eqref{4.1} of the transition matrices
$A_l=\Phi^{-1}\Psi$ and $A_r=\Psi^{-1}\Phi$, we compute
\begin{align*}
[A_l]_t&=\Phi^{-1}\Psi_t-\Phi^{-1}\Phi_t\Phi^{-1}\Psi\\
&=\Phi^{-1}\Psi_t-\Phi^{-1}\Phi_tA_l\\
&=\Phi^{-1}\left(\bT\Psi-\Psi[C_\Psi]_tC_\Psi^{-1}\right)
-\Phi^{-1}\left(\bT\Phi-\Phi[C_\Phi]_tC_\Phi^{-1}\right)A_l\\
&=\Phi^{-1}\bT\Psi-A_l[C_\Psi]_tC_\Psi^{-1}
-\Phi^{-1}\bT A_l+[C_\Phi]_tC_\Phi^{-1}A_l\\
&=[C_\Phi]_tC_\Phi^{-1}A_l-A_l[C_\Psi]_tC_\Psi^{-1}=k^2(A_l(k)A-AA_l(k)).
\end{align*}
Consequently, using that $A_r(k;t)=A_l(k;t)^{-1}$ we obtain
\begin{equation}\label{6.5}
A_l(k;t)=e^{-k^2tA}A_l(k;0)e^{k^2tA},\qquad
A_r(k;t)=e^{-k^2tA}A_r(k;0)e^{k^2tA}.
\end{equation}
Since $A=\frac{i}{n}\text{diag}(I_{m_+},1-n,I_{m_-})$, we see that the diagonal
and corner blocks of $A_l(k;t)$ and $A_r(k;t)$ are time independent whereas the
remaining four blocks are not. More precisely,
\begin{equation}\label{6.6}
A_{l,r}(k;t)=\begin{bmatrix}A_{l,r}(k;0)_{++}&e^{-ik^2t}A_{l,r}(k;0)_{+0}
&A_{l,r}(k;0)_{+-}\\e^{ik^2t}A_{l,r}(k;0)_{0+}&A_{l,r}(k;0)_{00}
&e^{ik^2t}A_{l,r}(k;0)_{0-}\\A_{l,r}(k;0)_{-+}&e^{-ik^2t}A_{l,r}(k;0)_{-0}
&A_{l,r}(k;0)_{--}\end{bmatrix}.
\end{equation}
With the help of \eqref{4.12} we obtain for the scattering matrices
\begin{subequations}\label{6.7}
\begin{align}
S_{l,r}(k;t)=\begin{bmatrix}S_{l,r}(k;0)_{++}&e^{-ik^2t}S_{l,r}(k;0)_{+0}
&S_{l,r}(k;0)_{+-}\\e^{ik^2t}S_{l,r}(k;0)_{0+}&S_{l,r}(k;0)_{00}
&e^{ik^2t}S_{l,r}(k;0)_{0-}\\S_{l,r}(k;0)_{-+}&e^{-ik^2t}S_{l,r}(k;0)_{-0}
&S_{l,r}(k;0)_{--}\end{bmatrix},\label{6.7a}\\
\tS_{l,r}(k;t)=\begin{bmatrix}\tS_{l,r}(k;0)_{++}&e^{-ik^2t}\tS_{l,r}(k;0)_{+0}
&\tS_{l,r}(k;0)_{+-}\\e^{ik^2t}\tS_{l,r}(k;0)_{0+}&\tS_{l,r}(k;0)_{00}
&e^{ik^2t}\tS_{l,r}(k;0)_{0-}\\\tS_{l,r}(k;0)_{-+}&e^{-ik^2t}\tS_{l,r}(k;0)_{-0}
&\tS_{l,r}(k;0)_{--}\end{bmatrix}.\label{6.7b}
\end{align}
\end{subequations}
As a result, in the absence of spectral singularities the Wiener-Hopf factors
$\A_{l,r}^\pm(k)$ in \eqref{4.16} are time independent as are the corrected
transmission coefficients $\CT(k)$ and $\tilde{\CT}(k)$ and the correction
factors $\CC(k)$ and $\tilde{\CC}(k)$. The corrected reflection coefficients
$\CR(k;t)$ and $\tilde{\CR}(k;t)$ are time varying and satisfy
$$\CR(k;t)=e^{-k^2tA}\CR(k;0)e^{k^2tA},\qquad
\tilde{\CR}(k;t)=e^{-k^2tA}\tilde{\CR}(k;0)e^{k^2tA}.$$

Let us now derive the time evolution of the norming constants $\bN_j$ and
$\tilde{\bN}_j$ under the assumption that the zeros $k_j$ of $A_l(k)_{++}$ and
$A_r(k)_{--}$ in $\C^+$ are simple and nonoverlapping, the zeros $\tilde{k}_j$
of $A_r(k)_{++}$ and $A_l(k)_{--}$ in $\C^-$ are simple and nonoverlapping, and
the winding numbers $w$ and $\tw$ both vanish. We now recall that $\bXi^+(x,k)$
and $\bXi^-(x,k)$ are $n\times n$ matrix solutions of \eqref{1.1} satisfying the
asymptotic conditions [cf. \eqref{4.8}]
\begin{subequations}\label{6.8}
\begin{align}
e^{-ikx\zS}\bXi^+(x,k;t)&\simeq\begin{cases}\ClD^\up_+(k;t),&x\to+\infty,\\
\ClD^\dn_+(k;t),&x\to-\infty,\end{cases}\label{6.8a}\\
e^{-ikx\zS}\bXi^-(x,k;t)&\simeq\begin{cases}\ClD^\dn_-(k;t),&x\to+\infty,\\
\ClD^\up_-(k;t),&x\to-\infty,\end{cases}\label{6.8b}
\end{align}
\end{subequations}
where the four matrices $\ClD^\up_\pm(l;t)$ and $\ClD^\dn_\pm(k;t)$ satisfy the
time evolutions
\begin{equation}\label{6.9}
\ClD^\up_\pm(k;t)=e^{-k^2tA}\ClD^\up_\pm(k;0)e^{k^2tA},\quad
\ClD^\dn_\pm(k;t)=e^{-k^2tA}\ClD^\dn_\pm(k;0)e^{k^2tA}.
\end{equation}
Hence,
\begin{subequations}\label{6.10}
\begin{align}
\bXi^+(x,k;t)&=\Psi(x,k;t)\ClD^\up_+(k;t)=\Phi(x,k;t)\ClD^\dn_+(k;t),
\label{6.10a}\\
\bXi^-(x,k;t)&=\Psi(x,k;t)\ClD^\dn_-(k;t)=\Phi(x,k;t)\ClD^\up_-(k;t).
\label{6.10b}
\end{align}
\end{subequations}
Moreover [cf. \eqref{4.12}],
\begin{subequations}\label{6.11}
\begin{align}
S(k;t)&=\ClD^\up_+(k;t)^{-1}\ClD^\dn_-(k;t)=\ClD^\dn_+(k;t)^{-1}\ClD^\up_-(k;x),
\label{6.11a}\\
\tilde{S}(k;t)&=\ClD^\up_-(k;t)^{-1}\ClD^\dn_+(k;t)
=\ClD^\dn_-(k;t)^{-1}\ClD^\up_+(k;x).\label{6.11b}
\end{align}
\end{subequations}
The time evolutions \eqref{6.7} are immediate from \eqref{6.9} and \eqref{6.10}.
 Finally, using \eqref{6.9} we obtain
\begin{subequations}\label{6.12}
\begin{align}
&[C_{\Xi^+}(k;t)]_tC_{\Xi^+}(k;t)^{-1}=\ClD^\up_+(k;t)^{-1}
\left(-k^2A\ClD^\up_+(k;t)-[\ClD^\up_+(k;t)]_t\right)\nonumber\\
&=e^{-k^2tA}\ClD^\up_+(k;0)^{-1}e^{k^2tA}\left(-k^2Ae^{-k^2tA}
\ClD^\up_+(k;0)e^{k^2tA}\right.\nonumber\\&-\left.e^{-k^2tA}
\left[-k^2A\ClD^\up_+(k;0)+\ClD^\up_+(k;0)k^2A\right]e^{k^2tA}\right)\nonumber\\
&=e^{-k^2tA}\ClD^\up_+(k;0)^{-1}\left(-k^2A\ClD^\up_+(k;0)+k^2A\ClD^\up_+(k;0)
-\ClD^\up_+(k;0)k^2A\right)e^{k^2tA}\nonumber\\
&=-e^{-k^2tA}\ClD^\up_+(k;0)^{-1}\ClD^\up_+(k;0)k^2Ae^{k^2tA}\nonumber\\
&=-e^{-k^2tA}(k^2A)e^{k^2tA}=-k^2A.\label{6.12a}
\end{align}
Similarly,
\begin{align}
[C_{\Xi^-}(k;t)]_tC_{\Xi^-}(k;t)^{-1}&=-k^2A.\label{6.12b}
\end{align}
\end{subequations}

Letting $\btau_j$ stand for the residue of $\tilde{\CT}(k)$ at $k=k_j$ and
$\tilde{\btau}_j$ for the residue of $\CT(k)$ at $k=\tilde{k}_j$, we define the
{\it norming constants} as those off-diagonal $3\times3$ matrices $\bN_j(t)$ and
$\tilde{\bN}_j(t)$ such that
\begin{subequations}\label{6.13}
\begin{align}
\bXi^+(x,k_j;t)\btau_j&=+i\,\bXi^+(x,k_j;t)\tilde{\CC}(k_j)^{-1}\bN_j(t),
\label{6.13a}\\
\bXi^-(x,\tilde{k}_j;t)\tilde{\btau}_j
&=-i\,\bXi^-(x,\tilde{k}_j;t)\CC(\tilde{k}_j)^{-1}\tilde{\bN}_j(t),\label{6.13b}
\end{align}
\end{subequations}
where $\btau_j$ and $\tilde{\btau}_j$ are time independent but the norming
constants are not.

Using \eqref{6.12a}, $[\bXi^+]_t=\bT\bXi^+-\bXi^+[C_{\bXi^+}]_tC_{\bXi^+}^{-1}$,
and the time independence of $\btau_j$ we get
\begin{align*}
&\left(\bT(k_j)\bXi^+(x,k_j;t)-\bXi^+(x,k_j;t)(-k_j^2A)\right)\btau_j
=i\left(\bT(k_j)\bXi^+(x,k_j;t)\right.\\&-\left.\bXi^+(x,k_j;t)(-k_j^2A)\right)
\tilde{\CC}(k_j)^{-1}\bN_j(t)+i\,\bXi^+(x,k_j;t)\tilde{\CC}(k_j)^{-1}[\bN_j]_t,
\end{align*}
which on deleting the (identical) first terms on either side leads to the
identity
$$\bXi^+(x,k_j;t)k_j^2A\btau_j=i\,\bXi^+(x,k_j;t)k_j^2A\tilde{\CC}(k_j)^{-1}
\bN_j(t)+i\,\bXi^+(x,k_j;t)\tilde{\CC}(k_j)^{-1}[\bN_j]_t,$$
thus implying
$$\bXi^+(x,k_j;t)\btau_jk_j^2A
=i\,\bXi^+(x,k_j;t)\tilde{\CC}(k_j)^{-1}k_j^2A\bN_j(t)
+i\,\bXi^+(x,k_j;t)\tilde{\CC}(k_j)^{-1}[\bN_j]_t.$$
With the help of \eqref{6.13a} we obtain in turn
\begin{align*}
i\,\bXi^+(x,k_j;t)\tilde{\CC}(k_j)^{-1}\bN_j(t)k_j^2A
&=i\,\bXi^+(x,k_j;t)\tilde{\CC}(k_j)^{-1}k_j^2A\bN_j(t)\\
&+i\,\bXi^+(x,k_j;t)\tilde{\CC}(k_j)^{-1}[\bN_j]_t.
\end{align*}
Therefore,
$$\bN_j(t)k_j^2A=k_j^2A\bN_j(t)+[\bN_j]_t,$$
and hence
$$[\bN_j]_t=\bN_j(t)k_j^2A-k_j^2A\bN_j(t).$$
Consequently, we have derived the time evolution of the norming constants
\begin{subequations}\label{6.14}
\begin{align}
\bN_j(t)&=e^{-k_j^2tA}\bN_j(0)e^{k^2tA}.\label{6.14a}\\
\intertext{In the same way we derive from \eqref{6.13b}}
\tilde{\bN}_j(t)&=e^{-\tilde{k}_j^2tA}\tilde{\bN}_j(0)e^{\tilde{k}_j^2tA}
\label{6.14b}
\end{align}
\end{subequations}

\medskip
{\bf Data availability statement:} In this article no numerical experiments were
conducted.

{\bf Conflicts of interest:} The author declares no conflict of interest.

{\bf Research funding:} The research leading to this aricle was conducted
without any resarch funding.

\appendix

\section{Linear equations for wedge products}\label{sec:A}

The exterior algebra of $\R^n$ is defined as the graded associate algebra
$$\zL(\R^n)=\bigoplus_{k=0}^n\,\zL^k(\R^n),$$
where $\zL^0(\R^n)=\R$, $\zL^1(\R^n)=\R^n$, and $\zL^k(\R^n)$ consists of the
finite linear combinations of all elements $\bv^1\wedge\ldots\wedge\bv^k$ for
$\bv^j\in\R^n$ ($j=1,2,\ldots,k$) (See \cite{MT}). Here the wedge product
$\wedge$ is multilinear, associative, and antisymmetric. The canonical basis of
$\zL^k(\R^n)$ is the set of $\binom{n}{k}$ entries
$$\left\{\be_{r_1}\wedge\ldots\wedge\be_{r_k}\right\}_{r_1<\ldots<r_k},$$
where $\{\be_1,\ldots,\be_n\}$ is the canonical basis of $\R^n$. Then
$\zL^{n-1}(\R^n)$ is an $n$-dimensional real vector space with bases
$\{\bd_1,\ldots,\bd_n\}$ and $\{\bg_1,\ldots,\bg_n\}$ related by the
equations \cite{MT}
\begin{equation}\label{A.1}
\bd_1=\bg_1,\qquad\bd_j=(-1)^{(j-1)(n-j)}\bg_j,\quad j=2,\ldots,n-1,
\qquad\bd_n=\bg_n,
\end{equation}
where, for $\bza_j=\be_1\wedge\ldots\wedge\be_{j-1}$ and
$\bzb_j=\be_{j+1}\wedge\ldots\wedge\be_n$, we have
\begin{align*}
\bd_j&=\bzb_j\wedge\bza_j
=\be_{j+1}\wedge\ldots\wedge\be_n\wedge\be_1\wedge\ldots\wedge\be_{j-1},\\
\bg_j&=\bza_j\wedge\bzb_j=\be_1\wedge\ldots\wedge\xcancel{\be_j}\wedge\ldots
\wedge\be_n.
\end{align*}
 For $n=3$ we can identify the wedge product with the vector product in $\R^3$
and obtain $\bg_j\simeq\be_j$ ($j=1,2,3$).

Writing $A=e^{xB}$ for $x\in\R$ and $B:\R^n\to\R^n$
another linear transformation, then the mapping $B^{(n)}_k$ defined by
\begin{align*}
&B^{(n)}_k(\bv^1\wedge\ldots\wedge\bv^k)
=\left.\frac{d}{dx}e^{xB}\bv^1\wedge\ldots\wedge e^{xB}\bv^k\right|_{x=0}\\
&=\left.\frac{d}{dx}(I_n+xB)\bv^1\wedge\ldots\wedge(I_n+xB)\bv^k\right|_{x=0}\\
&=\sum_{j=1}^k\,(I_n+xB)\bv_1\wedge\ldots\wedge(I_n+xB)\bv^{j-1}\wedge
B\bv^j\\ &\phantom{\sum_{j=1}^k\sum_{j=1}^k}\left.\wedge(I_n+xB)\bv^{j+1}\wedge
\ldots\wedge(I_n+xB)\bv^k\right|_{x=0}\\
&=\sum_{j=1}^k\,\bv^1\wedge\ldots\bv^{j-1}\wedge B\bv^j\wedge\bv^{j+1}\wedge
\ldots\wedge\bv^k
\end{align*}
is a linear transformation on $\zL^k(\R^n)$ and hence can be represented as
an $\binom{n}{k}\times\binom{n}{k}$ matrix with respect to a suitable basis. In
the particular case $k=n-1$ and with respect to the basis
$\{\bd_1,\ldots,\bd_n\}$ we get $A^{(n)}_{n-1}=\text{cofac}(A^T)$, the transpose
cofactor matrix of $A$. For the nonsingular matrix $A=e^{xB}$ we get
\begin{equation}\label{A.2}
B^{(n)}_{n-1}=\left.\frac{d}{dx}(\det e^{xB})e^{-xB^T}\right|_{x=0}
=(\text{Tr}\,B)I_n-B^T.
\end{equation}

Let us work out the wedge products
\begin{subequations}\label{A.3}
\begin{align}
(-1)^{m_+m_-}&\left(W^T\be_1\wedge\ldots\wedge W^T\be_{m_+}\wedge\be_{n+1-m_-}
\wedge\ldots\wedge\be_n\right)\nonumber\\&=\sum_{j=1}^{m_++1}(-1)^{j+m_+m_-}
\begin{vmatrix}W_{1,1}&\ldots&W_{m_+,1}\\ \vdots&&\vdots\\
\xcancel{W_{1,j}}&\ldots&\xcancel{W_{m_+,j}}\\ \vdots&&\vdots\\
W_{1,m_++1}&\ldots&W_{m_+,m_++1}\end{vmatrix}\bd_j,\label{A.3a}
\end{align}
\begin{align}
(-1)^{m_+m_-}&\left(\be_1\wedge\ldots\wedge\be_{m_+}\wedge W^T\be_{n+1-m_-}
\wedge\ldots\wedge W^T\be_n\right)\nonumber\\&=\sum_{j=m_++1}^n(-1)^{j+m_+m_-}
\begin{vmatrix}W_{n+1-m_-,m_++1}&\ldots&W_{n,m_++1}\\ \vdots&&\vdots\\
\xcancel{W_{n+1-m_-,j}}&\ldots&\xcancel{W_{n,j}}\\ \vdots&&\vdots\\
W_{n+1-m_-,n}&\ldots&W_{n,n}\end{vmatrix}\bd_j,\label{A.3b}
\end{align}
\end{subequations}
as linear combinations of $\bd_1,\ldots,\bd_n$, where $W^T=e^{ikx\bN_-}Z^T$ in
\eqref{A.3a} and $W^T=e^{-ik\bN_+}Z^T$ in \eqref{A.3b} with $Z^T$ independent of
$x\in\R$.

Singling out the coefficient of $\bd_{m_++1}$ we obtain
\begin{subequations}\label{A.4}
\begin{align}
\ClD_+^\up(k)_{00}&=-\det A_l(k)_{++}
=-\begin{vmatrix}A_l(k)_{1,1}&\ldots&A_l(k)_{1,m_+}\\ \vdots&&\vdots\\
A_l(k)_{m_+,1}&\ldots&A_l(k)_{m_+,m_+}\end{vmatrix},\label{A.4a}\\
\ClD_+^\dn(k)_{00}&=-\det A_r(k)_{--}
=-\begin{vmatrix}A_r(k)_{n+1-m_-,n+1-m_-}&\ldots&A_r(k)_{n+1-m_-,n}\\
\vdots&&\vdots\\ A_r(k)_{n,n+1-m_-}&\ldots&A_r(k)_{n,n}\end{vmatrix},
\label{A.4b}\\
\ClD_-^\dn(k)_{00}&=-\det A_l(k)_{--}
=-\begin{vmatrix}A_l(k)_{n+1-m_-,n+1-m_-}&\ldots&A_l(k)_{n+1-m_-,n}\\
\vdots&&\vdots\\ A_l(k)_{n,n+1-m_-}&\ldots&A_l(k)_{n,n}\end{vmatrix},
\label{A.4c}\\
\ClD_-^\up(k)_{00}&=-\det A_r(k)_{++}
=-\begin{vmatrix}A_r(k)_{1,1}&\ldots&A_r(k)_{1,m_+}\\ \vdots&&\vdots\\
A_r(k)_{m_+,1}&\ldots&A_r(k)_{m_+,m_+}\end{vmatrix},\label{A.4d}
\end{align}
\end{subequations}
where the factors $e^{ikx\bN_-}$ and $e^{-ikx\bN_+}$ do not affect the result.

Identifying $\bd_j$ with $\be_j$ we obtain
\begin{subequations}\label{A.5}
\begin{align}
&\be_j^T\ClD_+^\up(k)_{+0}=(-1)^{j+m_+m_-}e^{ikx}\times\nonumber\\
&\times\begin{vmatrix}
A_l(k)_{1,1}&\ldots&\xcancel{A_l(k)_{1,j}}&\ldots&A_l(k)_{1,m_+}
&A_l(k)_{1,m_++1}\\ \vdots&\vdots&&\vdots&\vdots\\ A_l(k)_{m_+,1}&\ldots
&\xcancel{A_l(k)_{m_+,j}}&\ldots&A_l(k)_{m_+,m_+}&A_l(k)_{m_+,m_++1}
\end{vmatrix},\label{A.5a}\\
&\be_j^T\ClD_-^\up(k)_{+0}=(-1)^{j+m_+m_-}e^{ikx}\times\nonumber\\
&\times\begin{vmatrix}
A_r(k)_{1,1}&\ldots&\xcancel{A_r(k)_{1,j}}&\ldots&A_r(k)_{1,m_+}
&A_r(k)_{1,m_++1}\\ \vdots&\vdots&&\vdots&\vdots\\ A_r(k)_{m_+,1}&\ldots
&\xcancel{A_r(k)_{m_+,j}}&\ldots&A_r(k)_{m_+,m_+}&A_r(k)_{m_+,m_++1}
\end{vmatrix},\label{A.5b}\\
\intertext{for $j=1,\ldots,m_+$, and}
&\be_j^T\ClD_-^\dn(k)_{-0}=(-1)^{j+m_+m_-}e^{-ikx}\times\nonumber\\
&\times\begin{vmatrix}
A_r(k)_{n+1-m_-,m_++1}&A_r(k)_{n+1-m_-,n+1-m_-}&\ldots
&\xcancel{A_r(k)_{n+1-m_-,j}}&\ldots&A_r(k)_{n+1-m_-,n}\\
\vdots&\vdots&&\vdots&&\vdots\\ A_r(k)_{n,m_++1}&A_r(k)_{n,n+1-m_-}&\ldots
&\xcancel{A_r(k)_{n,j}}&\ldots&A_r(k)_{n,n}
\end{vmatrix},\label{A.5c}\\
&\be_j^T\ClD_+^\dn(k)_{-0}=(-1)^{j+m_+m_-}e^{-ikx}\times\nonumber\\
&\times\begin{vmatrix}
A_l(k)_{n+1-m_-,m_++1}&A_l(k)_{n+1-m_-,n+1-m_-}&\ldots
&\xcancel{A_l(k)_{n+1-m_-,j}}&\ldots&A_l(k)_{n+1-m_-,n}\\
\vdots&\vdots&&\vdots&&\vdots\\ A_l(k)_{n,m_++1}&A_l(k)_{n,n+1-m_-}&\ldots
&\xcancel{A_l(k)_{n,j}}&\ldots&A_l(k)_{n,n}
\end{vmatrix},\label{A.5d}
\end{align}
\end{subequations}
for $j=n+1-m_-,\ldots,n$.

\section{Wiener algebras}\label{sec:B}

By the (continuous) Wiener algebra $\CW$ we mean the complex vector space of
constants plus Fourier transforms of $L^1$-functions
$$\CW=\{c+\hat{h}:c\in\C,\ h\in L^1(\R)\}$$
endowed with the norm $|c|+\|h\|_1$. Here we define the Fourier transform as
follows: $(\CF h)(k)=\hat{h}(k)=\int_{-\infty}^\infty dy\,e^{iky}h(y)$.
The invertible elements of the commutative Banach algebra $\CW$ with unit
element are exactly those $c+\hat{h}\in\CW$ for which $c\neq0$ and
$c+\hat{h}(k)\neq0$ for each $k\in\R$ \cite{Gel,GRS,Kn}.

The algebra $\CW$ has the two closed subalgebras $\CW^+$ and $\CW^-$ consisting
of those $c+\hat{h}$ such that $h$ is supported on $\R^+$ and $\R^-$,
respectively. The invertible elements of $\CW^\pm$ are exactly those
$c+\hat{h}\in\CW^\pm$ for which $c\neq0$ and $c+\hat{h}(k)\neq0$ for each
$k\in\C^\pm\cup\R$ \cite{Gel,GRS,Kn}. Letting $\CW^\pm_0$ and $\CW_0$ stand for
the (nonunital) closed subalgebras of $\CW^\pm$ and $\CW$ consisting of those
$c+\hat{h}$ for which $c=0$, we obtain the direct sum decompositions
$$\CW=\C\oplus\CW^+_0\oplus\CW^-_0,\qquad\CW_0=\CW^+_0\oplus\CW^-_0.$$

By $\Pi_\pm$ we now denote the (bounded) projections of $\CW$ onto $\CW^\pm_0$
along $\C\oplus\CW^\mp_0$. Then $\Pi_+$ and $\Pi_-$ are complementary
projections. In fact,
\begin{equation}\label{B.1}
(\Pi_\pm f)(k)=\frac{1}{2\pi i}\int_{-\infty}^\infty d\zz\,
\frac{f(\zz)}{\zz-(k\pm i0^+)},
\end{equation}
where $f\in\CW_0\cap L^p(\R)$ for some $p\in(1,+\infty)$. These direct sum
decompositions coupled with the Fourier transform can be schematically represented
as follows:
$$\begin{array}{ccccc}
L^1(\R)&=&L^1(\R^-)&\oplus&L^1(\R^+)\\
\Big\downarrow\vcenter{\rlap{$\CF$}}&&\Big\downarrow\vcenter{\rlap{$\CF$}}&&
\Big\downarrow\vcenter{\rlap{$\CF$}}\\
\CW_0&=&\CW^-_0&\oplus&\CW^+_0
\end{array}$$

Now observe that $\CF$ acts as an isometric linear one-to-one correspondence
from $L^1(\R)$ onto $\CW_0$. If we define the norm of $\C\oplus L^1(\R)$ as
$\|c+h\|=|c|+\|h\|_1$, we obtain the direct sum decomposition
$$L^1(\R)=L^1(\R^+)\oplus L^1(\R^-),$$
where the projection $\CF^{-1}\Pi_\pm\CF$ is the restriction of an arbitrary
$h\in L^1(\R)$ to the half-line $\R^\pm$.

The following result is most easily proved using the Gelfand theory of
commutative Banach algebras \cite{Gel,GRS,Kn}.

\begin{theorem}\label{th:B.1}
If for some complex number $h_\infty$ and some $h\in L^1(\R)$ the Fourier
transform $h_\infty+\int_{-\infty}^\infty dz\,e^{ikz}h(z)\neq0$ for every
$k\in\R$ and if $h_\infty\neq0$, then there exists $\zk\in L^1(\R)$ such that
$$\frac{1}{h_\infty+\int_{-\infty}^\infty dz\,e^{ikz}h(z)}
=\frac{1}{h_\infty}+\int_{-\infty}^\infty dz\,e^{ikz}\zk(z)$$
for every $k\in\R$.
\end{theorem}

\medskip
{\bf Disclaimer/Publisher's Note:} The statements, opinions, and data contained
in all publications are solely those of the individual author(s) and
contributor(s) and not of MDPI and/or editor(s). MDPI and or the editor(s)
disclaim responsability for any injurys to people or property from any ideas,
methods, instructions or products referred to in the content.


\begin{thebibliography}{WW}
\bibitem{AC} M.J. Ablowitz and P.A. Clarkson, {\it Solitons, Nonlinear
Evolution Equations and Inverse Scattering}, Cambridge University Press,
Cambridge, 1991.
\bibitem{AKNS} M.J. Ablowitz, D.J. Kaup, A.C. Newell, and H. Segur, {\it The
inverse scattering transform -- Fourier analysis for nonlinear problems},
Stud. Appl. Math. {\bf 53}, 249--315 (1974).
\bibitem{APT} M.J. Ablowitz, B. Prinari, and A.D. Trubatch, {\it Discrete and
Continuous Nonlinear Schr\"o\-din\-ger Systems}, Cambridge University Press,
Cambridge, 2004.
\bibitem{AS} M.J. Ablowitz and H. Segur, {\it Solitons and the Inverse
Scattering Transform}, SIAM, Philadelphia, 1981.
\bibitem{BC1} R. Beals and R.R. Coifman, {\it Scattering and inverse
scattering for first order systems}, Commun. Pure Appl. Math. {\bf 37}(1),
39-90 (1984).
\bibitem{BC2} R. Beals and R.R. Coifman, {\it Inverse scattering and evolution
equations}, Commun. Pure Appl. Math. {\bf 38}(1), 29--42 (1985).
\bibitem{BDT} R. Beals, P. Deift, and C. Tomei, {\it Direct and Inverse
Scattering on the Line}, Math. Surveys {\bf 28}, Amer. Math. Soc., Providence,
RI, 1988.
\bibitem{CD} F. Calogero and A. Degasperis, {\it Spectral Transforms and
Solitons}, North-Holland, Amsterdam, 1982.
\bibitem{YON2} M. Caso-Huerta, A. Degasperis, P. Leal da Silva, S. Lombardo, and
M. Sommacal, {\it Periodic and solitary wave solutions of the long-short-wave
Yajima-Oikawa-Newell model}, Fluids {\bf 7}, 227 (2022), 15~pp.
\bibitem{YON} M Caso-Huerta, A. Degasperis, S. Lombardo, and M. Sommacal,
{\it A new integrable model of long wave-short wave interaction and linear
stability spectra}, Proc. Roy. Soc. London A {\bf 477}(2252), 20210408 (2021).
\bibitem{DTT} P. Deift, C. Tomei, and E. Trubowitz, {\it Inverse scattering and
the Boussinesq equation}, Commun. Pure Appl. Math. {\bf 35}(5), 567--628 (1982).
\bibitem{DM24} F. Demontis and C. van der Mee, {\it Three-way Zakharov-Shabat
systems with zero diagonal entry}, Ricerche di Matematica {\bf 2025}, 100206
(2025).
\bibitem{FT} L.D. Faddeev and L.A. Takhtajan, {\it Hamiltonian Methods in the
Theory of Solitons}, Classics in Mathematics, Springer, New York, 1987; also:
Nauka, Moscow, 1986 [Russian].
\bibitem{Gel} I.M. Gelfand, {\it Normierte Ringe}, Mat. Sbornik {\bf 9}, 3--24
(1941).
\bibitem{GRS} I.M. Gelfand, D.A. Raikov, and G.E. Shilov, {\it Commutative
Normed Rings}, Chelsea Publ., New York, 1964; also: Fizmatgiz, Moscow, 1960
[Russian].
\bibitem{GF} I.C. Gohberg and I.A. Feldman, {\it Convolution Equations and
Projection Methods for their Solution}, Transl. Math. Monographs {\bf 41},
Amer. Math. Soc., Providence, RI, 1971; also: Nauka, Moscow, 1971 [Russian].
\bibitem{Kn} E. Kaniuth, {\it A Course in Commutative Banach Algebras},
Graduate Texts in Mathematics {\bf 246}, Springer, New York and Berlin, 2009.
\bibitem{Kp2} D.J. Kaup, {\it On the inverse scattering problem for cubic
eigenvalue problems of class} $\psi_{xxx}+6Q\psi_x+6R_\psi=\zl\psi$, Stud. Appl.
Math. {\bf 62}(3), 189--216 (1980).
\bibitem{Kr} M.G. Krein, {\it Integral equations on the half-line with a kernel
depending on the difference of arguments}, Uspehi Mat. Nauk {\bf 13}(5), 3--120
(1958) [Russian].
\bibitem{Ln} D. Lannes, {\it The Water Waves Problem\/}: {\it Mathematical
Analysis and Asymptotics}, Math. Surveys and Monographs {\bf 188}, Amer. Math.
Soc., Providence, RI, 2013.
\bibitem{Geng3} Ruomeng Li and Xianguo Geng, {\it A matrix Yajima-Oikawa
long-wave-short-wave resonance equation, Darboux transformations and rogue wave
solutions}, Commun. Nonlinear Sci. Numer. Simulat. {\bf 90}, 105408 (2020).
\bibitem{Ma} Yan-Chow Ma, {\it The complete solution of the long-wave-short-wave
resonance equations}, Stud. Appl. Math. {\bf 59}(3), 201--221 (1978).
\bibitem{MT} I. Madsen and J. Tornehave, {\it From Calculus to Cohomology},
Cambridge University Press, Cambridge, 1997.
\bibitem{McK} H.P. McKean, {\it Boussinesq's equation as a Hamiltonian system}.
In: I. Gohberg and M. Kac, {\it Topics in Functional Analysis. Essays Dedicated
to M.G. Krein on the Occasion of his 70th Birthday}, Academic Press, New York,
1978, pp.~217--226.
\bibitem{Nw} A.C. Newell, {\it Long waves--short waves: A solvable model}, SIAM
J. Appl. Math. {\bf 35}(4), 650--664 (1978).
\bibitem{PAB} B. Prinari, M.J. Ablowitz, and G. Biondini, {\it Inverse
scattering transform for the vector nonlinear Schr\"o\-din\-ger equation with
nonvanishing boundary conditions}, J. Math. Phys. {\bf 47}, 063508 (2016),
pp.~33.
\bibitem{WCGL} Kedong Wang, Mingming Chen, Xianguo Geng, and Ruomeng Li,
{\it A vector super Newell long-wave-short-wave equation and infinite
conservation laws}, Partial Differential Equations in Applied Mathematics
{\bf 5}, 100206 (2022).
\bibitem{Wh} G.B. Whitham, {\it Non-linear dispersion of water waves}, J. Fluid
Mech. {\bf 27}(2), 399--412 (1966).
\bibitem{Wr} O.C. Wright III, {\it Homoclinic connections of unstable plane
waves of the long-wave-short-wave equations}, Stud. Appl. Math. {\bf 117},
71--93 (2006).
\bibitem{YO} N. Yajima and M. Oikawa, {\it Formation and interaction of
Sonic-Langmuir solitons}, Prog. Theor. Phys. {\bf 56}(6), 1719--1739 (1976).
\bibitem{ZS} V.E. Zakharov and A.B. Shabat, {\it Exact theory of
two-dimensional self-focusing and one dimensional self-modulation of waves in
nonlinear media}, Sov. Phys. JETP {\bf 34}, 62--69 (1972); also: Z. Eksper.
Teoret. Fiz. {\bf 61}(1), 118--134 (1971) [Russian].
\end{thebibliography}
\end{document}